\documentclass[centertags,12pt]{amsart}

\usepackage{latexsym}
\usepackage{amsthm}
\usepackage{amssymb}
\usepackage{color}
\usepackage[utf8]{inputenc}
\usepackage{mathrsfs}
\usepackage[english]{babel}
\usepackage{setspace}
\usepackage{textcomp}
\usepackage{graphicx}
\graphicspath{ {images/} }
\usepackage{cancel}

\numberwithin{equation}{section}
\makeatletter
\renewcommand{\subsection}{\@startsection
{subsection}{2}{0mm}{\baselineskip}{-0.25cm}
{\normalfont\normalsize\bf}}
\makeatother

\newtheorem{theorem}{Theorem}[section]
\newtheorem{proposition}[theorem]{Proposition}
\newtheorem{lemma}[theorem]{Lemma}
\newtheorem{corollary}[theorem]{Corollary}

   {\theoremstyle{definition}

\newtheorem{question}[theorem]{Question}

\newtheorem{example}[theorem]{Example}
}

   \theoremstyle{remark}
\newtheorem{remark}[theorem]{Remark}

\newcommand{\F}{{\mathbb F}}

\newcommand{\cX}{{\mathcal X}}

\newcommand{\cZ}{{\mathcal Z}}
\newcommand{\cF}{{\mathcal F}}
\newcommand{\cH}{{\mathcal H}}

\newcommand{\cC}{{\mathcal C}}
\newcommand{\cD}{{\mathcal D}}

\newcommand{\aut}{{\rm Aut}}

\begin{document}

\author[D. Bartoli]{Daniele Bartoli}
\address{Dipartimento di Matematica e Informatica, 
Universit\`a degli Studi di Perugia,
via Vanvitelli 1, 06123 Perugia, Italy}
\email{daniele.bartoli@unipg.it}

\author[M. Montanucci]{Maria Montanucci}
\address{Dipartimento di Matematica Informatica ed Economia, Universit\`a degli Studi della Basilicata,
Campus di Macchia Romana, viale dell'Ateneo Lucano 10, 85100 Potenza, Italy}
\email{maria.montanucci@unibas.it}

\author[F.Torres]{Fernando Torres}
\address{IMECC/UNICAMP, R. S\'ergio Buarque de Holanda 651, Cidade
Universit\'aria  ``Zeferino Vaz", 13083-859, Campinas,
SP-Brazil}
\email{ftorres@ime.unicamp.br}

\thanks{{\em 2010 Math. Subj. Class.}: Primary 11G; Secondary 14G}


\title[Maximal curves Galois-covered by the Hermitian curve]{$\mathbb{F}_{p^2}$-maximal 
curves with many automorphisms are Galois-covered by the Hermitian curve}

  \begin{abstract} Let $\F$ be the finite field of order $q^2$, $q=p^t$ with $p$ prime. It is 
sometimes attribute to J.P. Serre the fact that any curve $\F$-covered by the Hermitian curve 
$\cH_{q+1}:\, y^{q+1}=x^q+x$ is also $\F$-maximal. Nevertheless, the converse is not true as the
Giulietti-Korchm\'aros example shows provided that $q>8$ and $t\equiv 0\pmod{3}$. In this 
paper, we show that if an $\F$-maximal curve $\cX$ of genus $g\geq 2$,  $q=p$, is such that  
$|\aut(\cX)|>84(g-1)$ then $\cX$ is Galois-covered by $\cH_{p+1}$. Also, we show that 
the hypothesis on the order of $\aut(\cX)$ is sharp, since there exists an $\mathbb{F}$-maximal 
curve $\cX$ for $q=p=71$ of genus $g=7$ with $|\aut(\cX)|=84(7-1)$ which is not Galois-covered 
by the Hermitian curve $\cH_{72}$.
   \end{abstract}

\maketitle

   \section{Introduction}\label{s1}

Throughout this paper, by a {\em curve} we shall mean a projective, non-singular, 
geometrically irreducible algebraic curve defined over a finite field $\F=\F_{q^2}$ of order 
$q^2$. A curve $\cX$ of genus $g=g(\cX)$ is called {\em $\F$-maximal} if the number of its 
$\F$-rational points attains the Hasse-Weil upper bound; that is to say
   $$ 
   |\cX(\F)|=q^2+1+2qg\, . 
   $$ 
Ihara \cite{Ih} proved that if $\cX$ is  $\F$-maximal of genus $g$ then $g\leq q(q-1)/2$. Also, equality holds if and only if $\cX$ is $\F$-isomorphic to the Hermitian curve, namely the plane curve 
${\cH}_{q+1}:\, y^{q+1}=x^q+x$; see R\"uck and Stichtenoth \cite{RS}. For surveys on maximal
curves we refer the readers to \cite{FT1, Garcia, Garcia1, GS1, VDG1, VDG2,FT, FT2, RS, SX}. 

Maximal curves have also been investigated for their applications in Coding theory. In fact, Algebraic-Geometric codes constructed from maximal curves have the largest possible minimum distances with the respect to other parameters since they have the maximum number of rational points compared with their genus; see \cite{
Hansen,Matthews1,Matthews2,Stichtenoth1,Stichtenoth2,Tiersma,XC2002,XL2000,YK1992}. Also, many examples of  maximal curves  have large automorphism groups and codes constructed from them inherit a large number of symmetries and therefore  can have good performance in encoding \cite{HLS1995} and decoding \cite{Joyner}.


In general it is quite difficult to prove the maximality of a given curve $\mathcal{C}$. A  well-known approach  is based on a  result of Kleiman \cite{Kleiman},  sometimes attributed to Serre (see \cite{Lachaud}), stating that any non-singular curve which is $\F$-covered by an $\F$-maximal curve is also $\F$-maximal. 
Concrete examples of $\F$-maximal curves which all are Galois-covered by $\cH_{q+1}$ can be found e.g. in \cite{GSX,CKT,MZq,MX}. 

For a long time it has been conjectured that all $\mathbb{F}$-maximal curves are covered by a Hermitian
curve; see e.g. \cite{Geer}.
This conjecture was disproved by Giulietti and Korchm\'aros \cite{GK,T^3}, who exhibited an example of $\mathbb{F}$-maximal curve $\cC$, with $q=p^{3h}>8$,  $p$ is a prime,  non-covered by the Hermitian curve. 
The curve $\cC$ has a large automorphism group whose size exceeds  the Classical Hurwitz bound  $84(g(\cC)-1)$.

Up to now the Giulietti-Korchm\'aros curve and some of its quotients are the only known examples of  $\F$-maximal curves which are not covered by the Hermitian curve; see \cite{T^3,GQZ}. It is somehow natural to ask whether there exist other curves with this feature,   also when $q \ne p^{3h}$ for any $h \geq 1$.  

In this paper we deal with the case $q=p$. Our main result is summarized in Theorem \ref{main}. We show that an $\mathbb{F}_{p^2}$-maximal curve with a large automorphism group and not Galois-covered by the Hermitian curve $\cH_{p+1}$  cannot exist. This emphasizes once more the importance of the Giulietti-Korchm\'aros curve among the class of maximal curves.

  \begin{theorem}\label{main} Let $p$ be a prime and $\F=\mathbb{F}_{p^2}$ the finite field 
of order $p^2.$ Let $\cX$ be an $\F$-maximal curve with genus $g=g(\cX) \geq 2,$ and 
$\aut(\cX)$ the $\F$-automorphism group of $\cX.$ If $\aut(\cX)$ does not satisfy the 
classical Hurwitz bound$;$ i$.$e$.,$ $|\aut(\cX)|> 84(g-1),$ then $\cX$ is Galois-covered by the 
Hermitian curve $\cH_{p+1}: y^{p+1}=x^p+x$ over $\F.$
  \end{theorem}
  
The  bound $Aut(\cX)>84(g-1)$ in Theorem \ref{main} is sharp. As we show in Theorem  \ref{nonga}  the so-called {\em Fricke-Macbeath} curve  over $\F_{{71}^2}$ is $\mathbb{F}$-maximal and not covered by $\mathcal{H}_{71}$. Such a curve is named after Robert Fricke who first studied it as Riemann surface in the early 1899; see \cite{Fricke}. Since the size of its automorphism group is exactly $84(g-1)$, this shows that the bound in  Theorem \ref{main} is sharp and it cannot be further improved in general.

   \section{Preliminary results}\label{s2}
   
Let $\cX$ be a curve defined over the finite field $\F=\F_{q^2}$ of order $q^2$ with $q=p^t$ 
a power of a prime $p$. Let $\aut(\cX)$ be the $\F$-automorphism group of $\cX$. For $m\geq 
1$ a divisor of $q+1$, we let $\cH_m$ denote the non-singular model of the plane curve
   $$
   y^m=x^q+x\, .
   $$
Notice that $\cH_{q+1}$ is the aforementioned Hermitian curve. We start by recalling a 
characterization of $\cH_m$ involving automorphisms of curves. This characterization is due to 
Garcia and Tafazolian, see \cite[Thm. 5.2]{GT} and \cite[Thm, 4.1]{TT}.
    \begin{theorem}\label{hm} Let $\cX$ be an $\F$-maximal curve$.$ 
  Suppose that there exists an abelian subgroup $H$ of $\aut(\cX)$ 
   whose order equals 
$q$ such that the quotient curve $\cX/H$ is rational$.$ Then there exists a divisor $m$ of 
$q+1$ such that $\cX$ is $\F$-isomorphic to the curve $\cH_m$ above$.$
 
   \end{theorem}
  \begin{lemma}\label{hmcovered} The curve $\cH_m$ above is Galois-covered over $\F$ 
  by the Hermitian curve $\cH_{q+1}.$ Also$,$ $\aut(\cH_m)$ contains a cyclic subgroup $C_m$ of 
  order $m$ such that $\aut(\cH_m)/C_m \cong PGL(2,q)$ and hence $|\aut(\cH_m)| >84(g-1),$ where 
  $g=g(\cH_m)=(m-1)(q-1)/2.$ 
  \end{lemma}
  \begin{proof} Let $\Sigma$ and $\Sigma'$ be the corresponding $\F$-function fields of 
  $\cH_{q+1}$ and $\cH_m$ respectively. Clearly $\Sigma'$ is $\F$-covered by 
  $\Sigma$, because of the morphism $\varphi: 
(x,y) \mapsto (x,y^{(q+1)/m})$; also, $[\Sigma:\Sigma']=(q+1)/m$. Consider the 
$\F$-automorphism group $G$ of $\Sigma$ given by 
  $$
  G=\{\varphi_{\lambda}: (x,y) \mapsto (x,\lambda y) \mid \lambda^{(q+1)/m}=1\}\, ,
   $$ 
   which is of order $(q+1)/m$. Then  $\Sigma'$ is the fixed field of $G$ on $\Sigma$, as the functions $x$ and $y^{(q+1)/m}$ 
are fixed by $G$ and $|G|=[\Sigma : \Sigma']$. The claim on the structure of $\aut(\cH_m)$ 
follows from \cite[Thm. 12.11]{HKT}.
   \end{proof}
Let $\cX$ be a curve over $\F$ of genus $g=g(\cX)$, $G$ a subgroup of $\aut(\cX)$ and $P\in 
\cX$.  We recall that an {\em orbit} of $G$ (resp. a {\em stabilizer in $G$}) of $P$ is the 
set $G(P):=\{\tau(P):\tau\in G\}$ (resp. $G_P:=\{\tau\in G: \tau(P)=P\}$). The orbit $G(P)$ 
is either {\em short} or {\em long} provided that $|G_P|>1$ or not. A short orbit $G(P)$ is 
either {\em tame} or {\em non-tame} according to $p\nmid |G_P|$ or not, where $p$ is the 
characteristic of $\F$.

The following theorem gives the exact structure of the short orbits of $G$ on $\cX$ 
when $|G|>84(g(\cX)-1)$; see \cite[Thm 11.56, Thm. 11.126]{HKT}.
   \begin{theorem}\label{thmHurwitz} Let $\cX$ be curve over $\F$ of genus $g \geq 2$ and 
let $G \leq \aut(\cX)$ with $|G|>84(g-1).$ Then the quotient curve $\cX/G$ is rational and 
$G$ has at most three short orbits on $\cX$ as follows$:$
   \begin{enumerate}
\item[\rm(1)] Exactly three short orbits$:$ one non-tame and two tame$.$ Each point in each tame 
short orbit has stabilizer in $G$ of order $2;$ 
\item[\rm(2)] Exactly two short orbits$,$ both non-tame$;$
\item[\rm(3)] Only one short orbit which is non-tame$;$
\item[\rm(4)] Exactly two short orbits$,$ one non-tame tame and one tame. In this case $|G|< 8g^3,$ with 
the following exceptions$:$
   \begin{itemize}
\item $p=2$ and $\cX$ is isomorphic to the hyperelliptic curve $y^2+y=x^{2^k+1}$ with 
  genus $2^{k-1};$ 
\item $p>2$ and $\cX$ is isomorphic to the Roquette curve $y^2=x^q-x$ with genus $(q-1)/2;$ 
\item $p\geq 2$ and $\cX$ is isomorphic to the Hermitian curve $y^{q+1}=x^q+x$ with 
genus $q(q-1)/2;$ 
\item $p=2$, $q_0=2^s$, $q=2q_0^2$ and $\cX$ is isomorphic to the Suzuki curve 
$y^q+y=x^{q_0}(x^q+x)$ with genus $q_0(q-1).$ 
   \end{itemize}
   \end{enumerate}
   \end{theorem}
The following lemma will be used to ensure that a Sylow $p$-subgroup of a non-tame 
automorphism group of an $\F$-maximal curve $\cX$ fixes exactly one 
$\F$-rational point of $\cX$.
  \begin{lemma}{\rm (\cite[Prop. 3.8, Thm. 3.10]{GSY})}\label{1actionp2} Let $\cX$ be an 
$\F$-maximal curve of genus $g\geq2.$ Then the automorphism group $\aut(\cX)$ fixes the set 
$\cX(\F)$ of $\F$-rational points$.$ Also, automorphisms of $\cX$ over the algebraic 
closure of $\F$ are always defined over $\F.$ 
   \end{lemma}
  \begin{corollary}\label{actionp2} Let $p$ denote the characteristic of the finite field $\F=\F_{q^2}$ where $q=p^t$ and let $\cX$ be an $\F$-maximal curve with 
$g=g(\cX) \geq 2$ such that $p \mid |\aut(\cX)|$. If $S$ is a Sylow $p$-subgroup of 
$\aut(\cX),$ then $S$ fixes exactly one point $P\in \cX(\F)$ and acts semiregularly on the 
set of the remaining $\F$-rational points of $\cX.$ In particular$,$ if $p \nmid g$, then every 
$\sigma \in S$ has order at most equal to $q.$
  \end{corollary}
  \begin{proof} From Lemma \ref{1actionp2}, $S$ acts on the set $\cX(\F)$ of $\F$-rational 
points of $\cX$. Since $|\cX(\mathbb{F})| \equiv 1 \pmod{p}$, $S$ must fix at least a 
point $P \in \cX(\mathbb{F})$. Also, $\cX$ has zero $p$-rank and hence the claim 
follows from \cite[Lemma 11.129]{HKT}.
   \end{proof}
The following lemma provides a characterization of the Hermitian curve $\cH_{p+1}$ in terms 
of the order of a Sylow $p$-subgroup of its full automorphism group.
  \begin{lemma}\label{psy1} Let $p$ be a prime and $\F$ a field of order $p^2.$ Let $\cX$ be 
an $\F$-maximal curve of genus $g$ such that there exists $G \leq \aut(\cX)$ with $p \mid 
|G|.$ Then we can write
  \begin{equation} \label{eq1}
g=\frac{a_1(p-1)}{2}+a_2p\, ,
  \end{equation}
where $a_1$ is a non-negative integer such that ${G_P}^{(a_1+1)}$ is the last non-trivial 
ramification group at a point $P \in \cX$ and $a_2=g(\cX/H),$ where $H$ is a subgroup of 
$G$ of order $p.$ Also, $p^2 \nmid |G|$ unless $\cX$ is $\F$-isomorphic to 
the Hermitian curve $\cH_{p+1}.$
  \end{lemma}
  \begin{proof} Let $H \leq G$ with $|H|=p$. From Corollary \ref{actionp2}, $H$ has exactly one 
  fixed point $P$ which is thus $\mathbb{F}$-rational. Clearly $H$ acts semiregularly on 
$\cX(\mathbb{F}) \setminus \{P\}$ and so $|H| \mid (p^2+2gp)$. From the 
Riemann-Hurwitz formula $2g-2=p(2a_2-2)+(a_1+2)(p-1)$, so that
   $$
   g=\frac{p(2a_2-2)+(a_1+2)(p-1)+2}{2}=\frac{a_1(p-1)}{2}+a_2p\, .
   $$
If $\cX \cong \cH_{p+1}$, there is nothing to prove; thus we assume that 
$\cX \not\cong \cH_{p+1}$. From \cite{RS}, this implies that $p^2+2gp<p^2+p^2(p-1)=p^3$, 
as $2g<p(p-1)$. Thus, if $S$ is a Sylow $p$-subgroup of $\aut(\cX)$ containing $H$, 
either $|S|=p^2$ or $S=H$. Assume that $|S|=p^2$. From the Riemann-Hurwitz formula
  $$
  2g-2=p^2(2g(\cX / S)-2)+(a_1+2)(p^2-1)+a_3p(p-1)\, ,
  $$
for some non-negative integer $a_3$. In fact, if $i,j \geq 1$ are such 
that $G_P^{(i+1)} \ne G_p^{(i)}$ and $G_P^{(j+1)} \ne G_p^{(j)}$ 
then $i-j \equiv 0 \mod p$; see \cite[Lemma 11.75 (v)]{HKT}. Thus,
  $$
  p(p-1)>2g=2g(\cX/S)p^2+a_1(p^2-1)+a_3p(p-1)\, .
  $$
By direct checking, since $a_1,a_3$ and $g(\cX/S)$ are non-negative, this implies that 
$g(\cX / S)=a_1=a_3=0$ and hence $g=0$, a contradiction.
  \end{proof}
In the following Lemma the known results on $\mathbb{F}$-maximal curves of high genus
are collected; see \cite{Ih, RS, FT, FT2, KT}.
   \begin{lemma}\label{resmaxim} Let $\cX$ be an $\F$-maximal curve of genus $g=g(\cX),$ where 
  $\F=\F_{q^2}.$ 
  \begin{enumerate} 
  \item[\rm(1)] $g\leq g_2:=\lfloor (q^2-q+4)/6\rfloor,$ or $g=g_1:=\lfloor(q-1)^2/4\rfloor,$ or 
  $g=g_0:=q(q-1)/2;$
  \item[\rm(2)] $g=g_0$ if and only if $\cX$ is $\F$-isomorphic to $\cH_{q+1};$
  \item[\rm(3)] $g=g_1$ if and only if $\cX$ is $\F$-isomorphic to $\cX_{(q+1)/2}$ (resp. $y^{q+1}=
  x^{q/2}+\ldots+x$) if $q$ is odd (resp. $q$ even)$.$ In particular$,$ in this case 
$\cX$ is a cyclic quotient of the Hermitian curve $\cH_{q+1}$ of order $2.$ 
  \end{enumerate}
  \end{lemma}
  \begin{corollary}\label{bounda1a2} Let $\cX$ be an $\F$-maximal curve of genus $g$ with 
$\F=\F_{p^2}$, $ p\geq 7$ a prime$.$ Let $a_1$ and $a_2$ be as in Lemma \ref{psy1} 
and suppose that one of the following conditions holds$:$
   \begin{enumerate}
\item[\rm(1)] $a_1>\lfloor (p^2-p+4)/3(p-1) \rfloor,$
\item[\rm(2)] $a_2>\lfloor (p^2-p+4)/6p \rfloor,$
\item[\rm(3)] $a_1+a_2 \geq (p-1)/2$ but $a_2 \leq \lfloor (p^2-p+4)/6p \rfloor.$
  \end{enumerate}
Then $\cX$ is Galois-covered by the Hermitian curve $\cH_{p+1}$.
\end{corollary}
  \begin{proof} If (1) or (2) holds then $g > \lfloor (p^2-p+4)/6\rfloor$, and the 
claim follows from Lemma \ref{resmaxim}. Assume that $a_1+a_2 \geq 
(p-1)/2$ but $a_2 \leq \lfloor (p^2-p+4)/6p \rfloor$. Then $a_1 \geq \lceil 
(p-1)/2 - (p^2-p+4)/6p \rceil= \lceil (2p^2-2p+4)/6p \rceil 
=(p-1)/3$. Hence, $g \geq (p-3)(p-1)/6+p > (p^2-p+4)/6$, and the claim 
follows again from Lemma \ref{resmaxim}.
  \end{proof}

   \section{Proof of Theorem \ref{main}}\label{s3}

Throughot this section, let $\cX$ be an $\F$-maximal curve of genus $g$, where $\F=\F_{p^2}$ 
with $p$ a prime. To prove Theorem \ref{main}, we first analyze the case in which $p \leq 5$. 
   \subsection{Case: $p \leq 5$}\label{s3.1}
Here we do not need the hypothesis $|\aut(\cX)|>84(g-1)$. For $p=2$ and $p=3$ the result is 
trivial by Lemma \ref{resmaxim}. For $p=5$ we use the complete classification, up to isomorphism, 
of $\F_{25}$-maximal curves given in \cite{FG}.
  \begin{lemma} \label{p5} Let $\cX$ be an $\F_{25}$-maximal curve of genus $g=g(\cX).$ 
  Then $\cX$ is Galois-covered by the Hermitian curve $\cH_6.$ 
  \end{lemma}
   \begin{proof} From \cite[Thm. 11]{FG}, $g(\cX) \in \{0,1,2,3,4,10\}$ where
\begin{enumerate}
\item $g=10$ if and only if $\cX \cong \cH_6: y^6=x^5+x$ over $\mathbb{F}_{25}$;
\item $g=4$ if and only if $\cX \cong \cH_3: y^3=x^5+x$ over $\mathbb{F}_{25}$;
\item $g=3$ if and only if $\cX \cong : \cC: y^6=x^5+2x^4+3x^3+4x^2+3xy^3$ over $\mathbb{F}_{25}$,
\item $g=2$ if and only if $\cX \cong \cH_2: y^2=x^5+x$ over $\mathbb{F}_{25}$,
\item$g(\cX)=1$ if and only if $\cX \cong \cD: y^2+x^3+1=0$ over $\mathbb{F}_{25}$.
\end{enumerate}
   The cases (2) and (4) are Galois covered by $\cH_6$ by Lemma \ref{hmcovered}. The elliptic 
   curve given in (5) is Galois covered by $\cH_6$ as this curve can also be 
   described by the Fermat equation $y^6+x^6+1=0$. Then $\cD=\cH_6/G$, where
   $$
   G=\{\alpha_{a,b}(x,y)=(ax,by) \mid a^{2}=b^{2}=1\} \cong \mathbb{Z}_2 \times \mathbb{Z}_2\, .
   $$ 
Finally, to prove that $\cC$ is also Galois covered by $\cH_6$, 
since there is an unique, up to isomorphism, $\F_{25}$-maximal curve of genus $3$ which is 
Galois covered by $\cH_6$ \cite[Thm. 5.6]{CKT}, it is sufficient to construct a quotient 
curve of $\cH_6$ of genus $3$. This follows considering the plane curve 
  $$
  \cZ:\quad x^5+y+2x^2y^2+xy^5=0\, .
  $$
From \cite[Prop. 2.1]{CKT}, the curve $\cZ$ equals $\cH_6/H$ where 
$H \leq \aut(\cH_6)$ is a subgroup of order $3$ of a Singer group of order $21$.
   \end{proof}
   \subsection{Case $p \geq 7$}\label{s3.2}
Here we assume that $\cX$ is an $\mathbb{F}_{p^2}$-maximal curve of genus $g \geq 2$ such that $|\aut(\cX)|>84(g-1)$ so that one of the cases listed in Theorem \ref{thmHurwitz} is satisfied. Let us start by 
showing that cases (1) and (2) in that result cannot occur.
   \begin{lemma}\label{type1} There is no an $\F_{p^2}$-maximal curve $\cX$ whose full 
automorphism group $\aut(\cX)$ satisfies any of the following property$:$
   \begin{enumerate}
   \item[\rm(1)] It admits exactly one non-tame short orbit $O_1$ and two tame short orbits 
$O_2$ and $O_3;$
   \item[\rm(2)] It admits exactly two non-tame short orbits $O_1$ and $O_2.$ 
   \end{enumerate}
   \end{lemma}
   \begin{proof} (1) By Theorem \ref{thmHurwitz}(4) $\cX$ is not isomorphic to $\cH_{p+1}$. 
In particular, from Lemma \ref{psy1} $p \mid |\aut(\cX)|$ but $p^2 \nmid |\aut(\cX)|$. 
Let $H$ be a $p$-Sylow subgroup of $\aut(\cX)$. From Corollary \ref{actionp2} $H$ fixes 
exactly an $\mathbb{F}_{p^2}$-rational point $P\in O_1$ and acts semiregularly in 
$O_1 \setminus \{P\}$. Thus we have that $|O_1|=1+np$ for some $n \geq 0$. From the 
Riemann-Hurwitz formula
  $$
  2g-2=|\aut(\cX)|(2 \cdot 0-2) + (1+np)[(|\aut(\cX)|/(1+np)-1)
  $$
  $$
  +(a_1+1)(p-1)]+\frac{|\aut(\cX)|}{2}[(2-1)+(2-1)]\, ,
  $$
and hence
   \begin{equation}\label{eq2}
   2g-2=(1+np)(a_1p+p-a_1-2)\, .
    \end{equation}
We now assume that $n>0$. From Lemma \ref{resmaxim} we have that 
$2g<p(p-1)$ as $\cX$ is not isomorphic to $\cH_{p+1}$. Thus, by direct checking, 
Equality \eqref{eq2} yields $n=1$ and $a_1=0$. In 
this case $2p(p-1)-2>2g-2=(p+1)(p-2)=2p(p-1)-2$, a contradiction. This yields that 
$n=0$ and $\aut(\cX)$ fixes a point $P \in \cX$. From \eqref{eq2} and \eqref{eq1}
   $$
   a_1(p-1)+2a_2p-2=2g-2=a_1(p-1)+(p-2)\, .
   $$
Since this implies that $2a_2p=p$, we have a contradiction.

(2) From Corollary \ref{actionp2} we know that the fixed points of the Sylow $p$-subgroups 
of $\aut(\cX)$ lie on $\cX(\mathbb{F}_{p^2})$, and hence $O_1$ and $O_2$ are contained in 
$\cX(\mathbb{F}_{p^2})$. Also, as before, the size of each non-tame short orbit of 
$\aut(\cX)$ is congruent to $1$ modulo $p$.

The size of the set $\cX(\mathbb{F}_{p^2}) \setminus O_1$ is congruent to $0 \pmod{p}$. As 
also $O_2 \subsetneq \cX(\mathbb{F}_{p^2})$ has length congruent to $1 \pmod{p}$ and 
$|\aut(\cX)|\equiv 0 \pmod{p}$ we have a contradiction. 
  \end{proof}
Now we deal with cases (3) and (4) of Theorem \ref{thmHurwitz}.
  \begin{lemma}\label{1orbit} Let $\cX$ be an $\F_{p^2}$-maximal curve and assume that 
$\aut(\cX)$ satisfies any of the following properties$:$
   \begin{enumerate} 
   \item[\rm(1)] It has exactly one non-tame short orbit $O_1;$
  \item[\rm(2)] It admits exactly one non-tame short orbit $O_1$ and one tame short orbit 
$O_2.$ 
   \end{enumerate}
Then $\cX$ is Galois-covered by $\cH_{p+1}.$
   \end{lemma}
   \begin{proof} (1) We can assume that for every Sylow $p$-subgroup $H$ of $\aut(\cX)$, 
the quotient curve $\cX/H$ is not rational; otherwise, the claim follows from Theorem 
\ref{hm}. In particular, this implies that $a_2>0$ in \eqref{eq1}. As before, write 
$|O_1|=1+np$ for $n \geq 0$. From \cite[Lemma 11.111]{HKT} we have that $|O_1|$ divides 
$2g-2$. Since the unique short orbit of $\aut(\cX)$ must be contained in 
$\cX(\mathbb{F}_{p^2})$ then either $\cX(\mathbb{F}_{p^2})=O_1$ or 
$\cX(\mathbb{F}_{p^2})=O_1 \cup (\bigcup_{i=1}^{t} \tilde O_i)$ where $| \tilde 
O_i|=|\aut(\cX)|$ for $t \geq 1$. Clearly the first case is not possible as 
$p^2+1+2gp>2g-2$. Assume that $\cX(\mathbb{F}_{p^2})=O_1 \cup (\bigcup_{i=1}^{t} \tilde 
O_i)$ where $| \tilde O_i|=|\aut(\cX)|$ for $t \geq 1$. This case can occur only if $|O_1|$ 
is a divisor of $p^2+1+2gp$ and $2g-2$, that is, only if $|O_1|$ is a divisor of $(p+1)^2$ 
which is congruent to $1$ modulo $p$. This implies three possible cases: 
either $|O_1|=(p+1)^2$\, (a) or $|O_1|=1$\, (b) or $|O_1|=p+1$\, (c).
  
{\bf (a) holds.} Here from Lemma \ref{resmaxim}, $p(p-1)-2>2g-2 \geq (p+1)^2$, 
a contradiction. 

{\bf (b) holds.} Here from the Riemann-Hurwitz formula
   $$
   2g-2=|\aut(\cX)|(-2)+[|\aut(\cX)|-1+(a_1+1)(p-1)]\, ,
   $$
and hence from \eqref{eq1} $2a_2p-p=-|\aut(\cX)|$, a contradiction to $a_2>0$.

{\bf (c) holds.} Since $\aut(\cX)$ has no other short 
orbits, we have that $(p+1) \mid p^2+1+2gp-(p+1)$ and hence $(p+1) \mid (a_1+1)(p-1)+2a_2p$. 
By direct computation this implies that $(p+1) \mid 2a_1+2a_2+2$. In particular $a_1+a_2 
\geq \frac{p-1}{2}$, while $1 \leq a_2 \leq \frac{p^2-p+4}{6p}$. The claim now follows from 
Lemma \ref{bounda1a2}.

(2) As before, we can assume that for every 
Sylow $p$-subgroup $H$ of $\aut(\cX)$ the quotient curve 
$\cX/H$ is not rational, otherwise the claim follows from Theorem \ref{hm}. 
In particular, this implies that $a_2>0$ in \eqref{eq1}. Assume that $|O_1|=1+np$ for 
$n \geq 0$.

\textbf{Case 1:} $\cX(\F_{p^2})=O_1 \cup O_2$. From the Riemann-Hurwitz formula
   $$
   2g-2=|\aut(\cX)|(-2)+(1+np)\Big[(a_1+1)(p-1)+\frac{|\aut(\cX)|}{1+np}-1\Big]
   $$
   $$
   +(p^2+1+2gp-1-np)\Big(\frac{|\aut(\cX)|}{p^2+1+2gp-1-np}-1\Big)\quad\text{and hence}\, ,
   $$
  $$
  2g-2=(1+np)[(a_1+1)(p-1)-1]-(p^2+(2g-n)p)\, .
  $$
Using \eqref{eq1} this reduces to $2a_2(p+1)=(a_1+1)(n-1)(p-1)$. 
Since $a_2>0$ and $(p-1,p+1)=2$, we have that $a_2 \geq (p-1)/4$ and 
hence $g \geq p(p-1)/2$. The claim now follows.

\textbf{Case 2:} $\cX(\F_{p^2})=O_1$. For $P \in O_2$ let $|\aut(\cX)_P|=hp$, 
where $(h,p)=1$. Then 
$h \leq 4a_2+2$, as $h$ is the order of a cyclic group in $\cX / H$, where $H$ is a 
$p$-group of order $p$; see \cite[Thm. 11.60]{HKT}. Let $Q \in O_2$. From the 
Riemann-Hurwitz formula
   $$
   2g-2=-2ph|O_1|+|O_1|((hp-1)+(a_1+1)(p-1))+|O_2|(|\aut(\cX)|/|O_2|-1)\, ;
   $$
or equivalently $|\aut(\cX)|=2(g-1)\frac{|\aut(\cX)_P| \cdot |\aut(\cX)_Q|}{N}$, 
where $N=|\aut(\cX)_Q|(-1+(a_1+1)(p-1))-|\aut(\cX)_P| \geq 1$; see 
\cite[(11.67) and (11.68)]{HKT}. This yields $-1/|\aut(\cX)_Q| 
\geq -(-1+(a_1+1)(p-1))/(hp+1)$ and hence
   $$
   \frac{2g-2}{ph|O_1|}=-\frac{1}{|\aut(\cX)_Q|}+
   \frac{(a_1+1)(p-1)-1}{hp} \geq \frac{(a_1+1)(p-1)-1}{hp(hp+1)}\, .
   $$
Thus,
   $$
   \frac{1}{hp^2} \geq \frac{2g-2}{2hgp^2} \geq \frac{2g-2}{ph|O_1|} \geq 
   \frac{(a_1+1)(p-1)-1}{hp(hp+1)}\, .
   $$
From the last inequalities, using $h \leq 4a_2+2$, we get that $(4a_2+2)p+1 \geq hp+1 \geq 
a_1p^2+p^2-a_1p-2p$ and hence
  $$
  a_2 \geq \frac{a_1p^2+p^2-a_1p-4p-1}{4p} \geq \frac{p^2-4p-1}{4p}\, ,\quad
  \text{while $g \geq  \frac{p(p^2-4p-1)}{4p}>g_3$}\, . 
  $$
If $p \geq 11$ this gives a contradiction from \cite{FT}. If $p=7$, then $g$ is not bigger 
than $g_3$ if and only if $a_1=0$. But since we are assuming that $a_2 \geq 1$, then $g \geq 
7=g_3$. The claim now follows from \cite[Thm. 5]{FGP}.

\textbf{Case 3:} $\cX(\F_{p^2})$ contains $O_1$ and at least a 
long orbit of $\aut(\cX)$. A case-by-case analysis is considered according to 
$n=0$, $n=1$ or $n>1$.
   
{\bf Assume that $n=0$.} In this case $O_1=\{P\}$ and 
$\aut(\cX)_P=\aut(\cX)$. 
From the Riemann-Hurwitz formula
   $$
   2g-2=-2|\aut(\cX)|+(a_1+1)(p-1)+|\aut(\cX)|-1+|O_2|(|\aut(\cX)|/|O_2|-1)\, .
   $$
Then $2a_2p-p=-|O_2|$, a contradiction since $a_2>0$.

{\bf Assume that $n=1$.} From the Riemann-Hurwitz formula
  $$
  2g-2=-2|\aut(\cX)| +(p+1)(|\aut(\cX)|/(p+1)-1+(a_1+1)(p-1))+|O_2|(|\aut(\cX)|/|O_2|-1)\, ,
   $$
and hence from \eqref{eq1}, $|O_2|=p(a_1+1)(p-1)-2a_2p$. 
The length $|O_2|$ must divide $p^2+2gp-p$. 
If $O_2$ is not contained in $\F_{p^2}$ then also $(p+1) \mid p^2+1+2gp-(p+1)$ and 
hence $(p+1) \mid (a_1+1)(p-1)+2a_2p$ and it divides $|\aut(\cX)|$. By direct 
computation this implies that $(p+1) \mid 2a_1+2a_2+2$. In particular $a_1+a_2 \geq 
\frac{p-1}{2}$ and so the claim follows from Lemma \ref{bounda1a2}. 
If $O_2$ is contained in $\cX(\F_{p^2})$, $p^2+1+2gp-(p+1)-|O_2|$ must be 
positive and $|O_2|$ divides $p^2+1+2gp-(p+1)-|O_2|=2a_2p(p+1)$. Also,
$|O_2/p|=(a_1+1)(p-1)-2a_2$ divides $p+2g-1=(a_1+1)(p-1)+2a_2p-1=|O_2|/p+2a_2+2a_2p-1$. This 
implies that $|O_2|/p$ divides both $2a_2(p+1)-1$ and $2a_2p(p+1)$ and hence $|O_2|=p$. In 
particular $(a_1+1)(p-1)-2a_2=1$ and $a_2 \geq (p-2)/2$. Since this implies that $g \geq 
p(p-2)/2$ the claim follows.

{\bf Assume that $n>1$.} From the Riemann-Hurwitz formula 
  $$
  2g-2=-2|\aut(\cX)|
  $$ 
  $$+(1+np)\Big(\frac{|\aut(\cX)|}{1+np}-1+(a_1+1)(p-1)\Big)+
  |O_2|\Big( \frac{|\aut(\cX)|}{|O_2|}-1\Big)\quad\text{and hence}\, ,
  $$ 
  $$
  |O_2|=p[-2a_2-n+1+n(a_1+1)(p-1)]\, .
  $$ 
Since $|O_2|$ is a divisor of $p(p+2g-n)$ we have that $|O_2| \leq (p^2-np)/2+gp$ and 
    $$
    \frac{p^2-p+4}{6}\geq g_3 \geq g \geq \frac{(1+np)(-1+(a_1+1)(p-1))}{p+2} - 
    \frac{p^2-np}{2(p+2)}+\frac{2}{p+2}\, ,
    $$ 
    which implies 
\begin{equation} \label{n6}
n \leq \Big\lfloor \frac{p^3+4p^2-4p+8-6a_1(p-1)}{6p^2-9p+6a_1p^2-6a_1p} \Big\rfloor\, .
\end{equation}

In particular we get that $a_1<p/6$ and $n<(p+6)/6$. 

Assume that $O_2$ is not contained in $\cX(\mathbb{F}_{p^2})$. 

As $1+np$ divides 
$p^2+1+2gp-1-np=p(p+2g-n)$, we have that $p+2g-n=p(1+a_1+2a_2)-a_1-n$. Write 
$p(1+a_1+2a_2)-a_1-n=\alpha(1+np)$. Then $\alpha\equiv -a_1-n\pmod{p}$ and 
since $a_1<p/6$, $n<(p+6)/6$ and $\alpha \leq p$, we get $\alpha=p-a_1-n$. Thus, 
$p(1+a_1+2a_2)-a_1-n=(p-a_1-n)(1+np)$, and hence
  $$
  \frac{p}{2}=\frac{p}{6}+\frac{2p}{6} \geq a_1+2a_2=n(p-a_1-n) \geq \frac{2p}{2}=p\, ,
  $$
a contradiction.

Thus, we can assume that $O_1$ and $O_2$ are both contained in $\cX(\F_{p^2})$.  

Since $1+np$ divides $p^2-np+2gp-|O_2|$, we have in particular that $1+np$ divides 
$p+(1-n)a_1(p-1)+2a_2(p+1)+n$ and hence 
  \begin{equation} \label{eqn}
n \leq 1+\frac{2a_2(p+1)}{(p-1)(a_1+1)}\, .
  \end{equation}
Write $k(1+np)=p+(1-n)a_1(p-1)+2a_2(p+1)+n$. Thus $k\equiv (n-1)a_1+2a_2+n\pmod{p}$ 
and $k \leq p$ as $p+(1-n)a_1(p-1)+2a_2(p+1)+n < p+\frac{p}{3}(p+1)+\frac{p+6}{6}=
(2p^2+9p+6)/6<p(1+2p)\leq p(1+np)$. We observe that from $a_2 \leq p/6$ and \eqref{eqn},
  $$
(n-1)a_1+2a_2+n \leq \frac{2a_2(p+1)a_1}{(p-1)(a_1+1)}+2a_2+1+\frac{2a_2(p+1)}{(p-1)(a_1+1)}
  $$
  $$ 
  \leq \frac{2a_2(p+1)}{(p-1)}+2a_2+1 \leq \frac{2a_2(2p)}{p-1}+1 < p+1\, ,
  $$
and hence $k=(n-1)a_1+2a_2+n$ with
  $$
  ((n-1)a_1+2a_2+n)(1+np)=p+(1-n)a_1(p-1)+2a_2(p+1)+n\, .
  $$
Thus $(n-1)((n+1)(a_1+1)+2a_2)=0$, which is impossible for $(n+1)(a_1+1)+2a_2 \geq 5$ 
and $n \ne 1$.
  \end{proof}
Thereferore the proof of Theorem \ref{main} follows from Theorem \ref{hmcovered} and Lemmas 
\ref{type1}, \ref{1orbit}.

  \section{On the hypothesis concerning the classical Hurwitz's bound}\label{s4}

Let $p$ be a prime and $\F$ be the finite field of order $p^2$. In view of Theorem 
\ref{main}, a natural question arises: Is any $\F$-maximal curve $\cX$ of genus $g=g(\cX)$ 
Galois-covered by $\cH_{p+1}$ also when the classical Hurwitz's bound 
\begin{equation}\label{HurwitzBound}
|\aut(\cX)| \leq 84(g-1)
\end{equation}
 hold true? As a matter of fact, it is not difficult to find examples of 
such curves which are not $\F$-isomorphic to $\cH_m: y^m=x^p+x$ with $m\mid (p+1)$.
  \begin{example}\label{es1} The curve $\cX$ given by the affine model over $\F=\F_{49}$
  $$
  y^8=x^4-x^2\, ,
  $$
  is $\F$-maximal with $g=g(\cX)=5$ \cite[Ex. 4.5.]{ATT}. For $m \mid 8$, we have that 
$g(\cH_m)=6(m-1)/2$. Thus, $\cX\ncong \cH_m$ for any $m \mid 8$. By direct 
checking using MAGMA (computational algebra system), $|\aut(\cX)|=192<336=84(g(\cX)-1)$. 
However we observe that $\cX$ is Galois-covered by $\cH_8$ from \cite[Ex. 6.4, Case 1]{GSX}.
  \end{example}
Next we present a $\F_{71}$-maximal curve of genus $7$ such that \eqref{HurwitzBound} holds 
which is not $\F$-Galois covered by $\cH_{72}$. The starting point is the 
plane equation over the complex numbers 
  $$
1+7xy+21x^2y^2+35x^3y^3+28x^4y^4+2x^7+2y^7=0
  $$
which defines the unique compact Riemann surface $\cF$ of genus $7$ such that 
$\aut(\cF)\cong PSL(2,8)$ (so that equality in \eqref{HurwitzBound} holds true); Fricke 
\cite{Fricke}, Macbeath \cite{MacBeath}, Edge \cite{Edge}, Hidalgo \cite{Hidalgo}. The curve 
$\cF$ is refered nowadays as the Fricke-Macbeath curve. This curve was considered over 
finite fields by Top and Verschoor \cite{TV}. In particular, we have the following.
   \begin{proposition} Let $p$ be a prime$,$ $p \equiv \pm 1 \pmod {14},$ and $\F=\F_{p^2}.$ 
   Then the Fricke-Macbeath curve $\cF$ above is $\F$-maximal if 
and only if the elliptic curve $\mathcal{E}:\, y^2-(x^3+x^2-114x-127)=0$ is 
$\F$-maximal$.$ 
   \end{proposition}
   \begin{proof} From \cite[Thm. 2.6]{TV}, since $p \equiv \pm 1 \pmod 7$, the number of 
$\mathbb{F}_{p^2}$-rational points of $\mathcal{F}$ satisfies 
$|\mathcal{F}(\mathbb{F}_{p^2})|=7|\mathcal{E}(\mathbb{F}_{p^2})|-6p^2-6$. Thus 
$|\mathcal{F}(\mathbb{F}_{p^2})|=p^2+1+14p$ if and only if 
$|\mathcal{E}(\mathbb{F}_{p^2})|=p^2+2p+1$ that is if and only if $\mathcal{E}$ is 
$\mathbb{F}_{p^2}$-maximal. 
   \end{proof}
  \begin{remark} By direct checking with MAGMA, the curve $\mathcal{F}$ is 
  $\mathbb{F}_{p^2}$-maximal 
for $p \in \{71,251,503,2591\}$ and $\aut(\cX) \cong PSL(2,8)$. Clearly, since the 
$\mathbb{F}_{p^2}$-maximality of $\mathcal{F}$ is equivalent to an elliptic curve to be 
supersingular there exist infinitely many values of $p$ for $\mathcal{F}$ to be 
$\mathbb{F}_{p^2}$-maximal, see \cite{Elkies}. 
  \end{remark}
The main reslt of this section is the following.
  \begin{theorem}\label{nonga} The $\cF$ Fricke-Macbeath curve above is 
  $\F_{71}^2$-maximal of genus $7$ with $\aut(\cF)\cong PSL(2,8)$ but it is not a 
Galois subcover of $\cH_{72}$.
  \end{theorem}
  \begin{proof} We only need to show the assertion on the covering. 
  The proof is long and very technical. Assume by contradiction that 
  $\mathcal{F} \cong \cH_{72}/G$ for some $G \leq 
PGU(3,71)$. The order of $PGU(3,71)$ is equal to $2^7 \cdot 3^5 \cdot 5 \cdot 7 \cdot 71 \cdot 
1657$. From the Riemann-Hurwitz formula, 
   $$ 
\frac{|\cH_{72}(\mathbb{F}_{71^2})|}{|\mathcal{F}(\mathbb{F}_{71^2})|} \leq |G| \leq 
\frac{2g(\cH_{72})-2}{2g(\mathcal{F})-2}\, ,
   $$ 
   which yields $60 \leq |G| \leq 414$, as 
$2g(\cH_{72})-2=4968$ and $2g(\mathcal{F})-2=12$. Since $|G|$ divides $|PGU(3,71)|$, 
we have to deal with 47 cases, namely
   \begin{equation} \label{ordini}
\begin{aligned}
|G| \in \{60,63,64,70,71,72,80,81,84,90,96,105,108,112,120,126, \\
128,135,140,142,144,160,162,168,180,189,192,210,213,216,224,\\ 240,243,252,270,280,
284,288,315,320,324,336,355,360,378,384,405\}\, . 
\end{aligned}
\end{equation}
The different divisor $\Delta$ has degree
   \begin{equation}\label{delta}
\deg(\Delta)=\sum_{\sigma \in G \setminus \{id\}} i(\sigma)=(2g(\cH_{72})-2)-
|G|(2g(\mathcal{F})-2)=4968-12|G|\, .
  \end{equation}
For the computation of $i(\sigma)$ we refer to the notation used in \cite[Lemma 2.2]{MZ} and the 
complete classification given in \cite[Thm. 2.7]{MZ}.

\textbf{Case 1: $71$ divides $|G|$}

{\bf $|G|=71$.} From \eqref{delta}, $\deg(\Delta)=4116$ but from \cite[Thm. 2.7]{MZ} either $\deg(\Delta)=70 \cdot 2$ or $\deg(\Delta)=70 \cdot (73)$, a contradiction.

{\bf $|G|=142$.} In this case either $G$ is dihedral or cyclic. From \eqref{delta}, 
$\deg(\Delta)=3264$. Assume that $G$ is dihedral. From \cite[Thm. 2.7]{MZ} either 
$\deg(\Delta)=70 \cdot 2 +71 \cdot 72 $ or $70 \cdot 73 +71$, a contradiction. If $G$ is 
cyclic then either $\deg(\Delta)=70 \cdot 2 +72+1 \cdot 70$ or $\deg(\Delta)=70 \cdot 73 + 
72 + 1 \cdot 70$, which are impossible.

{\bf $|G|=213$.} From \eqref{delta}, $\deg(\Delta)=2412$ and $G$ is cyclic. From \cite[Lemma 
2.2]{MZ}, if $\sigma \in G$ is tame then $\sigma$ is of type (A) and hence 
$i(\sigma)=p+1=72$. From \cite[Thm. 2.7]{MZ} either $\deg(\Delta)=70 \cdot 2 + 2 \cdot 72 + 
140 \cdot 1$ or $\deg(\Delta)=70 \cdot 73 + 2 \cdot 72 + 140 \cdot 1$, a contradiction.

{\bf $|G|=284$.} There are $4$ groups of order $284$ up to isomorphism. We will refer to 
such groups keeping the standard GAP notation as $G \cong SmallGroup(284,i)$ for 
$i=1,2,3,4$. From \cite[Thm. 2.7]{MZ}, if $G \cong SmallGroup(284,1)$ then $\deg(\Delta)=70 
\cdot 1 + 142 \cdot \alpha + 70 \cdot \beta + 72$, where $\alpha \in \{0,72\}$ and $\beta 
\in \{2,73\}$.  If $G \cong SmallGroup(284,2)$ then $\deg(\Delta)=140 \cdot 1 + 70 \cdot 1+ 
2 \cdot \alpha + 70 \cdot \beta + 72$, where $\alpha \in \{0,72\}$ and $\beta \in \{2,73\}$.  
If $G \cong SmallGroup(284,3)$ then $\deg(\Delta)= 70 \cdot 1+ 70 \cdot \beta + 143 \cdot 
72$, where $\beta \in \{2,73\}$. If $G \cong SmallGroup(284,2)$, then $\deg(\Delta)=210 
\cdot 1 + 70 \cdot \beta + 3 \cdot 72$, where $\beta \in \{2,73\}$. In all these cases, by 
direct checking $\deg(\Delta)$ does not satisfies \eqref{delta}, contraddicition.

{\bf \item $|G|=355$.} There are $2$ groups of order $355$ up to isomorphism, namely $G 
\cong C_{71} \rtimes C_5$ or $G \cong C_{355}$, where $C_n$ denotes a cyclic group of order 
$n$. In the former case $\deg(\Delta)=70 \cdot \alpha + 284 \cdot 2$, where $\alpha \in 
\{2,73\}$. By direct checking $\deg(\Delta)$ satisfies \eqref{delta} if and only if 
$\alpha=2$. Thus from the Riemann-Hurwitz formula $g(\cH_{72} / G)=7=g(\mathcal{F})$. 
Geometrically, the elements of $C_{71}$ have exactly one fixed point $P \in \cH_{72}$ while 
the elements of $C_5$ fix exactly the $\mathbb{F}_{71^2}$-rational vertexes of a triangle 
$T=\{P,Q,R\}$, where $Q \in \cH_{72}$ and $R \not\in \cH_{72}$. Let $H \cong 
C_{71^2-1}<PGU(3,71)$ fixing $T$ point-wise. Clearly $C_5<C_{71^2-1}$ and since $H$ 
normalizes $C_{71}$ $\tilde H=\langle C_{71},H \rangle =C_{71} \rtimes H$. Since $PSL(2,8)$ 
contains no subgroups of order $|\tilde H/G|$ the curves $\mathcal{F}$ and $\cH_{72}/G$ are 
not isomorphic. 

This shows that if $\cH_{72}/G \cong \mathcal{F}$, then $G$ must be tame.

\textbf{Case 2: $71$ does not divide $|G|$} We proceed with a case-by-case analysis 
according to $|G|$ and the possible group theoretical structure of $G$ up to isomorphisms. 
In most of the obtained cases a numerical contradiction to \eqref{delta} is obtained 
using \cite[Thm. 2.7]{MZ}. Here we underline again that, considering subgroups $G$ of 
$PGU(3,71)$ with $|G|$ satisfying one of the cases classified in \eqref{ordini}, 
quotient curves $\cH_{72}/G$ of $\cH_{72}$ of genus $7$ can be obtained. However, in each of 
these cases there exists at least a subgroup of $N_{PGU(3,71)}(G)/G$ which is not 
isomorphic to any subgroup of $PSL(2,8)$. 
The following is the complete list of quotient curves $\cH_{72}/G$ where $|G|$ 
satisfies \eqref{ordini} and $g(\cH_{72}/G)=7$.
   
  {\bf $|G|=72$ and $G \cong C_{72}$.} In this case $C_{71}=\langle \alpha \rangle$ where  $\alpha$ is either of type (A) or of type (B1) from \cite[Lemma 2.2]{MZ}. 
In both cases we can assume up to conjugation that $\cH_{72}$ is given by the Fermat equation 
$\cH_{72}: x^{72}+y^{72}+z^{72}=0$ and $\alpha$ admits 
a diagonal matrix representation of type 
   $$
   \alpha= [a,b,1]=\begin{pmatrix} a & 0 & 0 \\ 0 &  b & 0 \\ 0 & 0 & 1\end{pmatrix}\, ,
   $$
where $o(a)$ and $o(b)$ divide $p+1=72$. Since $\cH_{72}/G$ inherits at least a cyclic 
diagonal group of order $72$  of type $[\gamma,1,1]$,  $[1,\gamma,1]$ or $[1,1,\gamma]$ for some $\gamma$ of order $72$, we conclude that $\cH_{72}/G$ is not isomorphic to $\mathcal{F}$, as $PSL(2,8)$ does not contains abelian groups of order $72$.

{\bf $|G|=72$ and $G \cong SmallGroup(72,\ell)$ where $\ell \in \{9,18,36\}$.} Arguing as in 
the previous case, we observe that $\aut(\cH_{72}/G)$ inherits a cyclic group of order $n$ 
where $n \mid 72$ and $m \geq 9$. This conflicts with $\aut(\cH_{72}/G)$ to be isomorphic to 
$PSL(2,8)$.

{\bf $|G|=180$ and $G \cong SmallGroup(180,4)$, that is 
$G =\langle \sigma \rangle \cong C_{(71^2-1)/28}$.} Since $\sigma$ is of type (B2) of 
\cite[Lemma 2.2]{MZ}, we can assume that up to conjugation  $\cH_{72}$ has equation 
$x^{71}z+xz^{71}=y^{72}$ and $\sigma$ fixes the vertexes of the fundamental triangle 
$T=\{(1:0:0), (0:1:0), (0:0:1)\}$. Thus, $\sigma$ is given by a matrix representation  
  $$
  \sigma = \begin{pmatrix} a^{72} & 0 & 0 \\ 0 &  a & 0 \\ 0 & 0 & 1\end{pmatrix}\, ,
  $$
where $a \in \mathbb{F}_{71^2}$ with $o(a)=o(\sigma)=180$; 
see \cite[Page 644 case 8]{HKT}. Consider the automorphism $\alpha$ of $\cH_{72}$ given by
    $$
\begin{pmatrix} \xi^{72} & 0 & 0 \\ 0 &  \xi & 0 \\ 0 & 0 & 1\end{pmatrix}\, ,
    $$
where $o(\xi)=71^2-1$. Thus, $G <\langle \alpha \rangle \cong C_{71^2-1}$, 
and hence $\langle \alpha \rangle /G \cong C_{28} < \aut(\cH_{72}/G)$. 
Since $PSL(2,8)$ has no cyclic subgroups of order $28$, 
the curves $\mathcal{H}$ and $\cH_{72}/G$ are not isomorphic.

{\bf $|G|=240$ and $G \cong C_5 \rtimes C_{48}$.} The center $Z(G)$ is cyclic of order $24$. 
Geometrically, $Z(G)=\langle \alpha \rangle$ where $\alpha$ is of type (A) in \cite[Lemma 
2.2]{MZ}. The center of $\alpha$ is given by the fixed common point $P \not\in \cH_{72}$ of 
$C_5$ and $C_{48}$. By direct checking, using again a matrix representation for $C_{48}$ as 
in the previous case, we observe that the entire $C_{71^2-1}<PGU(3,71)$ containing $C_{48}$ 
normalizes $C_5$ as well. This yields the quotient curve $\cH_{72}/G$ to admit a cyclic 
group of automorphisms of order greater than $9$, a contradiction.

{\bf $|G|=240$ and $G \cong C_{240}$.} Arguing as in the case $|G|=180$ we observe that 
$\cH_{72}/G$ admits a cyclic automorphisms group of order $(71^2-1)/240=21$. Since 
$PSL(2,8)$ has no cyclic subgroups of order $21$, the curves $\mathcal{F}$ and $\cH_{72}/G$ 
are not isomorphic.

{\bf $|G|=315$ and $G \cong SmallGroup(315,2)$.} As $\cH_{72}/G$ admits at least a cyclic 
automorphisms group of order $(71^2-1)/315$, the claim follows.

{\bf $|G|=336$ and $G \cong SmallGroup(336,6)$.} In this case a contradiction is obtained 
observing that the quotient curve $\cH_{72}/G$ inherits a cyclic automorphisms group of 
order at least $15$.

{\bf $|G|=324$ and $G \cong SmallGroup(324,81)$.}  In this case a contradiction is obtained 
observing that the quotient curve $\cH_{72}/G$ inherits a cyclic automorphisms group of 
order at least $16$.

We now proceed with a case-by-case analysis for those cases for which a numerical 
contradiction to the Riemann-Hurwitz formula is obtained.

{\bf $|G|=60$.} In this case $\deg(\Delta)=4248=59 \cdot 72$. Since $G$ contains exactly 
$59$ non-trivial elements whose contribution to $\deg(\Delta)$ is at most $72$, every 
non-trivial element of $G$ is a homology and hence in particular $o(\sigma) \mid (q+1)$ for 
every $\sigma \in G \setminus \{id\}$; see \cite[Thm. 2.7]{MZ}. Since $5 \mid |G|$ and $5 
\nmid (p+1)$, this case is not possible.

{\bf $|G|=63$.} In this case $\deg(\Delta)=4214$, and $G_i \cong SmallGroup(63,i)$ for 
$i=1,\ldots,4$. Also, $G_1=\{12_{12},9_{42}, 7_{6},3_2,1_1\}$, $G_2=\{63_{36},21_{12}, 
9_6,7_{6},3_2\}$, $G_3=\{21_{12},7_{6},3_{44},1_1\}$, $G_4=\{21_{48},7_{6},3_{8},1_1\}$, 
where $n_m$ means that there are $m$ elements of order $n$ in the group. By \cite[Thm. 
2.7]{MZ} this gives $\deg(\Delta) \leq3204$, $\deg(\Delta) \leq 684$, $\deg(\Delta) \leq 
3204$, $\deg(\Delta) \leq 684$ respectively, a contradiction.

{\bf $|G|=64$.} Since every $\sigma \in G$ is a $2$-elements, by \cite[Thm. 2.7]{MZ} we have 
that $i(\sigma) \in \{0,2,72\}$. Hence we can write $\deg(\Delta)=4200$ as $72 \cdot i + 2 
\cdot j$ for some $0 \leq i+j \leq 63$. Such $i$ and $j$ do not exist and we have a 
contradiction.

{\bf $|G|=70$.} Arguing as in the previous case, we can write $\deg(\Delta)=4128=72 \cdot i + 2 \cdot j$ for some $0 \leq i+j \leq 69$. By direct computation with MAGMA, the unique possibility is $(i,j)=(57,12)$ and $G \cong SmallGroup(70,k)$ for $k=1,\ldots,4$. Since $70=p-1$, by \cite[Thm. 2.7]{MZ} the elements $\sigma \in G$ such that $i(\sigma)=72$ are those of order equal to $2$. Thus $i$ equals the number of involutions in $G$. If $G \cong SmallGroup(70,1)$ then $i=5$, if $G \cong SmallGroup(70,2)$ then $i=7$, if $G \cong SmallGroup(70,3)$ then $i=35$,  and if $G \cong SmallGroup(70,4)$ then $i=1$. Since in all cases $i \ne 57$ this case cannot occur.

{\bf $|G|=72$.} In this case $G \cong SmallGroup(72,a)$ for $a=1,\ldots, 50$, and 
$\deg(\Delta)=4104$ can be written as $72 \cdot i + 3 \cdot j$ for some $0 \leq i+j \leq 71$ 
by \cite[Thm. 2.7]{MZ}. By direct checking with MAGMA $(i,j)=(57,0)$ thus $G$ does not 
contains Singer subgroups. We consider the remaining cases according to the previous results 
obtained for groups of order $72$. We discard those cases for which $G$ contains more than 
$57$ involutions, which implies $i>57$.

Assume that $G \cong SmallGroup(72,1)$. Since $G$ has a unique involution, which is a 
homology, $G$ fixes an $\mathbb{F}_{71^2}$-rational point $P$ with $P \not\in \cH_{72}$. 
This implies that $G$ is contained in the maximal subgroup $\mathcal{M}_{71}$ of $PGU(3,71)$ 
fixing a $\mathbb{F}_{71^2}$-rational point off $\cH_{72}$. The center of $G$ is cyclic of 
order $4$ and it is generated by a homology. In fact assume by contradiction that $Z(G)$ is 
generated by an element $\gamma$ of type (B1). The elements $\alpha \in G$ of odd order 
commute with $\gamma$ and they fix a common point which is the center of the unique 
involution of $G$.  Thus, $\alpha$ fixes the fixed points of $\gamma$. This implies that the 
entire group $G$ fixes the fixed points of $\gamma$, and hence $G$ fixes pointwise a 
self-polar triangle $T$ with respect to the unitary polarity defined by $\cH_{72}$, $G \leq 
C_{72} \times C_{72}=Stab_{PGU(3,71)}(T)$ and $G$ is abelian, a contradiction. Thus $\gamma$ 
is a homology. From \cite[Page 6]{MZq}, $Z(\mathcal{M}_{71})$ is a cyclic group of order 
$71$ which is generated by a homology of center $P$. This implies that every element $\beta 
\in G \setminus Z(G)$ such that $\langle \beta \rangle$ intersects non-trivially $Z(G)$ is 
of type (B1) since otherwise $\beta \in Z(\mathcal{M}_{71})$ and hence $\beta \in Z(G)$, a 
contradiction. Looking at the subgroups structure of $G$ we get that $G$ contains at most 
$72-2-36=34$ homologies, so this case cannot occur.

If $G \cong SmallGroup(72,i)$, $i=3,4,5,6,8,10,11,12,13,14,16,20,27,28,30,47$, then arguing 
as above we get that $G$ contains at most $35,10,45,41,47,47,35,47,9,9,47,45,17,56,47,47$ 
homologies respectively, a contradiction.

If $G \cong SmallGroup(72,7)$ or $G \cong SmallGroup(72,17)$ then $G$ normalizes three 
distinct subgroups of order $2$, and hence fixes their centers. Then $G$ fixes the vertexes 
of self-polar triangle $T$ with respect to the unitary polarity defined by $\cH_{72}$ but 
$G$ is not abelian. Such a subgroup does not exist.
 
If $G \cong SmallGroup(72,15)$ then two cases are distinguished depending on the unique 
element of $\alpha\in G$ of order $3$ being a homology or not. By direct checking with 
MAGMA, if $\alpha$ is a homology then $\alpha \in Z(G)$. Since $Z(G)$ is trivial, this case 
cannot occur. Thus $\alpha$ is of type (B1): in this cases $G$ contains at most $71-2-24$ 
homologies, a contradiction.
  
If $G \cong SmallGroup(72,19)$. Since $G$ normalizes $7$ groups of order $2$, we have that 
$G$ has $7$ fixed points $P_1,\ldots,P_7$ which are $\mathbb{F}_{71^2}$-rational but not in 
$\cH_{72}$. This proves that in particular every element of $G$ is a homology, a 
contradiction.

The cases $G \cong SmallGroup(72,21)$ and $G \cong SmallGroup(72,22)$ cannot occur as in 
this case the unique subgroup of $G$ of order $3$ must be central, a contradiction.

If $G \cong SmallGroup(72,\ell)$ with $\ell \in \{29,32,33,34,35,37,48,49,50\}$, then $G$ fixes the vertexes of a self-polar triangle $T$ with respect to the unitary polarity defined by 
$\cH_{72}$ but $G$ is not abelian, a contradiction. 

The case $G \cong SmallGroup(72,\ell)$ with $\ell \in \{39,40,41\}$ cannot occur as a 
subgroups of $PGU(3,71)$. In fact $G$ contains a unique elementary abelian subgroup of order 
$9$ which is made by $6$ homologies and $2$ elements of type (B1). Thus elements of order 
$3$ cannot be all conjugate, a contradiction.

The cases $G \cong SmallGroup(72,\ell)$ with $\ell \in \{42,43,44\}$ cannot occur since at 
least one involution of $G$ must be central, a contradiction. The cases $G \cong 
SmallGroup(72,45)$ and $G \cong SmallGroup(72,46)$ cannot occur since at least one element 
of order $3$ in $G$ must be central, a contradiction.

{\bf $|G|=80$.} In this case we can write $\deg(\Delta)=4008=72 \cdot i + 2 \cdot j$, for $0 
\leq i+j \leq 79$. By direct checking with MAGMA the unique possibility is $(i,j)=(55,24)$. 
From \cite[Thm. 2.7]{MZ}, the elements $\sigma$ of $G$ for which $i(\sigma)=72$ are those 
with $o(\sigma) \in \{2,4,8\}$. Forcing $G \cong SmallGroup(80,k)$ to have at least $55$ 
elements of order in $\{2,4,8\}$, we get $k \in \{28,29,30,31,32,33,34,50\}$. Denote by 
$o_i$ the number of elements $\alpha \in G$ such that $o(\alpha)=i$. By direct checking with 
MAGMA we obtain that in each of these cases $o_2+o_4+o_8=63$ and hence $j=79-63=16$, a 
contradiction.

{\bf $|G|=81$.} In this case $\deg(\Delta)=3996=72 \cdot i + 3 \cdot j$ for some $0 \leq i+j 
\leq 80$. By direct checking with MAGMA the unique possibility is $(i,j)=(55,12)$. This 
proves that $G$ must contain at least $12$ elements of order $3$, as the unique elements 
$\beta$ in $G$ with $i(\beta)=3$ are those with $o(\beta)=3$ from \cite[Thm. 2.7]{MZ}. We 
note that these elements cannot be contained in cyclic groups of order $9$, as Singer groups 
of order $9$ does not exists in $PGU(3,71)$ because $9 \nmid (71^2-71+1)$. This yields $G 
\cong SmallGroup(81,k)$ with $k \in \{1,2,3,4\}$. The case $G \cong SmallGroup(81,1)$ cannot 
occur, as $G$ is cyclic of order not a divisor of $71^2-71+1$ and hence $G$ cannot contains 
Singer subgroups. Case $G \cong SmallGroup(81,2)$ cannot occur, as from a composition of 
elements of type (B1) or homologies it is not possible to generate Singer subgroups. Case $G 
\cong SmallGroup(81,3)$ cannot occur, as elements of order $9$ split in $3$ conjugacy 
classes giving rise to at least $18$ elements of order $3$ of type (B1). Hence $j \ne 12$.  
Case $G \cong SmallGroup(81,4)$ cannot occur, since elements of order $3$ generates a 
subgroup of $G$ isomorphic to $C_3 \times C_3$ but Singer subgroups do not commute.

{\bf $|G|=84$.} Since $84\mid (71^2-1)$, by \cite[Thm. 2.7]{MZ} we have that $i(\sigma) \in 
\{0,2,3,72\}$ for any $\sigma \in G$. Writing $\deg(\Delta)=3960=71 \cdot i + 2 \cdot j + 3 
\cdot k$ for $0 \leq i+j+k \leq 83$ we get that $(i,j,k)=(54,3m,2(12-m))$ for some 
$m=0,\ldots,5$ or $(i,j,k)=(55,0,0)$. We note that $(i,j,k)=(55,0,0)$ cannot occur as $G$ 
contains at least an element $\gamma$ with $o(\gamma)=7$ for which $i(\gamma)=2$ from 
\cite[Thm. 2.7]{MZ}. This implies that $(i,j,k)=(54,3m,2(12-m))$ for some $m=0,\ldots,5$ and 
hence $G$ contains at least $54+14=68$ elements $\alpha$ such that $o(\alpha) \in 
\{2,3,4,6,12\}$. By direct checking with MAGMA, $G \cong SmallGroup(84,\ell)$ where $\ell 
\in \{1,2,7,9,11\}$. For all these cases a contradiction is obtained combining the value of 
the parameter $k$ with the lengths of the conjugacy classes of elements of order $3$.

{\bf $|G|=90$.} We argue as in the previous cases. For $\sigma \in G$ we have $i(\sigma) \in 
\{0,2,3,72\}$. We can write $\deg(\Delta)=3008=72 \cdot i + 2 \cdot j+3 \cdot k$, for $0 
\leq i+j+k \leq 90$. Also, we observe that $j \leq 4$ since $G$ contains at least $4$ 
elements of order $5$. By direct computation with MAGMA we obtain that 
$(i,j,k)=(53,3m,2(12-m))$ for some $m=2, \ldots, 11$. Forcing $|\{\sigma \in G \mid 
o(\sigma)=2,3,6,9,18\}| \geq 53$ we get $G \cong SmallGroup(90,\ell)$ where $\ell \in 
\{7,9,11\}$. Assume that $G \cong SmallGroup(90,7)$. Hence $|\{\sigma \in G \mid 
o(\sigma)=2,3,6,9,18\}|=71$ and $j=18=3 \cdot 6$. This yields $k=2(12-6)=12$ which is 
impossible, as there exists a unique conjugacy class of elements of order $3$ whose length 
is not equal to $12$. Assume that $G \cong SmallGroup(90,9)$ or $G \cong SmallGroup(90,11)$. 
Hence $|\{\sigma \in G \mid o(\sigma)=2,3,6,9,18\}|=59$ and $j=30=3 \cdot 10$. This yields 
$k=2(12-10)=4$ which is impossible, as there exists a unique conjugacy class of elements of 
order $3$ whose length is not equal to $4$.

{\bf $|G|=96$.} Here for $\sigma \in G$ we have $i(\sigma) \in \{0,2,3,72\}$ and 
$\deg(\Delta)=3816=72 \cdot i + 2 \cdot j + 3 \cdot k$ for $0 \leq i+j+k \leq 95$. By direct 
computation with MAGMA $(i,j,k)=(52,3m,2(12-m))$ for some $m=0,\ldots,12$ or 
$(i,j,k)=(53,0,0)$. Forcing $\tau=|\{\sigma \in G \mid o(\sigma)=2,3,4,6,8,12,24\}| \geq 52$ 
we have $G \cong SmallGroup(96,\ell)$ with $\ell \in 
\{3,6,7,\ldots,17,20,\ldots,58,61,\ldots,231\}$. We note that elements of order $3$ in $G$ 
are all conjugated as $9 \nmid |G|$. This implies that if $k \ne 0$ then $k$ is exactly the 
number $o_3$ of elements of order $3$ in $G$. Assume that $G \cong SmallGroup(96,3)$. Then 
$\tau=95$ and hence $j=0=3 \cdot 0$. This implies that $k=2(12-0)=24$, since the number of 
elements of order $3$ in $G$ is not equal to $24$ we have a contradiction. Using the same 
argument we can exclude all the remaining cases for $G$. In fact if $\ell=6,7,8$ then 
$\tau=71$ and $k=8 \ne o_3$, if $\ell=9,\ldots,58$ then $\tau=95$ and $k=24 \ne o_3$, if 
$\ell=61,62,63$ then $\tau=71$ and $k=8 \ne o_3$, while if $\ell \geq 64$ then $\tau=71$ and 
$k=8 \ne o_3$.

{\bf $|G|=105$.} We write $\deg(\Delta)=3708=72 \cdot i + 2 \cdot j + 3 \cdot k$ where $10 
\leq i+j+k \leq 104$ as $j \geq 10$ because $G$ admits at least $4$ elements of order $5$ 
and $6$ elements of order $7$. By direct checking with MAGMA either 
$(i,j,k)=(50,3m,2(18-m))$ for some $m=4,\ldots,18$ or $(i,j,k)=(51,3m,2(6-m))$ for some 
$m=4,\ldots,6$. Since the unique elements $\alpha$ in $G$ for which $i(\alpha)$ can be equal 
to $72$ are those of order $3$ and their number is always less then $50$, this case is not 
possible.

{\bf $|G|=108$.} For $\sigma \in G$ we have $i(\sigma) \in \{0,2,3,72\}$ and elements of 
order $27, 54$ or $108$ do not exist as these integers do not occur as element orders in 
$PGU(3,71)$. We can write $\deg(\Delta)=3672=72 \cdot i + 2 \cdot j+3 \cdot k$, for $0 \leq 
i+j+k \leq 107$. By direct checking with MAGMA, $(i,j,k)=(49,3m,2(24-m))$ for some 
$m=0,\ldots,10$ or $(i,j,k)=(50,3m,2(12-m))$ for some $m=0,\ldots,12$ or $(i,j,k)=(51,0,0)$. 
All these conditions force $G \cong SmallGroup(108,\ell)$ with $\ell=6,\ldots,45$.

If $G \cong SmallGroup(108,6)$ then $j=0$ and hence $k \in \{48,24,0\}$. Since $G$ contains 
$8$ elements of order $3$ we get that $k=0$ and $(i,j,k)=(51,0,0)$. Here $G$ contains $36$ 
elements of order $12$ which are all conjugated. They are all of type (B1) because otherwise 
$G$ contains at least $36+18+1+2>51$ homologies, a contradiction. The center $Z(G) \cong 
C_6$ is generated by a homology as it has to commute with $6$ groups of order $12$ which are 
of type (B1). Also, $G$ contains a normal cyclic group of order $18$ which contains the 
central involution $\omega \in Z(G)$. If such a normal subgroup is generated by a homology 
then it must be contained in $Z(G)$, a contradiction. The same holds for the elements of 
order $6$ and $3$ contained in this normal subgroup. Thus $G$ contains at most 
$107-36-2-2-6=61$ homologies. Arguing in the same way for the other elements of order $18$ 
contained in $G$ we get that $G$ contains at most $61-12=49<51$ homologies, a contradiction.

The case $G \cong SmallGroup(108,7)$ cannot occur as a subgroup of $PGU(3,71)$. In fact 
$Z(G) \cong C_{18}$ and it is generated by homologies as it has to commute with the elements 
of order $36$ which are of type (B2) from \cite[Thm. 2.7]{MZ}. Clearly, the elements of 
order $12$ contained in cyclic groups of order $36$ are homologies as they are powers of 
elements of type (B2). This implies that there exist at least $6$ elements of order $12$ 
which are contained in $Z(G)$, a contradiction.

Assume that $G \cong SmallGroup(108,8)$. Thus, $j=0$ and $k=0$ since $G$ contains exactly 
$26$ elements of order $3$ which are normalized by involutions. This yields 
$(i,j,k)=(52,0,0)$. Arguing as in the previous case, we get that elements of order $12$ and 
$4$, which are all conjugated, are of type (B1), as they contain a central involution not 
belonging to the center of $G$. Thus $G$ contains at most $107-4 \cdot 9 - 9 \cdot 2$ 
homologies. On the other hand, there exists no cyclic subgroup of $G$ of order $6$ generated 
by homology, containing a central involution, and not contained in the center of $G$ and 
therefore $i\neq52$.

If either $G \cong SmallGroup(108,9)$, $G \cong SmallGroup(108,10)$ or $G \cong 
SmallGroup(108,11)$ then a contradiction is obtained arguing as in the previous case.

Denote by $o_i$ the number of elements having order equal to $i$ in $G$. We can exclude all 
the cases for which $o_2+o_3+o_4+o_6+o_9+o_{12}+o_{18}=107$, $o_2=1$ and $o_6>2$. This 
allows us to consider just the cases $G \cong SmallGroup(108,\ell)$ with $\ell \in 
\{12,14,15,\ldots,31,36,\ldots,45\}$.

Assume that $G \cong SmallGroup(108,12)$. Then $G$ fixes an $\mathbb{F}_{71^2}$-rational 
point $P$ with $p \not\in \cH_{72}$ as $G$ has a central involution. Since $G$ contains 
elements of type (B2), $Z(G)$ must be cyclic and generated by a homology. Thus, this case 
cannot occur.

Assume that $G \cong SmallGroup(108,14)$. Since $j=36$ we get that $(i,j,k)=(50,36,0)$. The 
elements of order $12$ contained in $Z(G)$ are homologies. The other elements of order $12$, 
the non-central elements of order $6$, and the elements of order $18$ in $G$ are of type 
(B1), as the cyclic groups that they generate have non-trivial intersection with $Z(G)$, a 
contradiction. This yields that $G$ contains at most $71$ homologies. To have $i=50$ we need 
$21$ homologies more, but this cannot happen as the lengths of the conjugacy classes of the 
remaining elements in $G$ are all even.

Assume $G \cong SmallGroup(108,15)$. Here $j=0$ and either $k=0$ and $i=51$ or $k=24$ and 
$i=50$. The elements of order $12$ in $G$ are of type (B1), because otherwise $G$ contains 
at least $36+18+1$ homologies, a contradiction. This yields $G$ to have at most $71$ 
homologies. Since $Z(G)$ commutes with all these elements of type (B1) it is generated by a 
homology. We note that the elements of order $6$ are of type (B1) because they are obtained 
as a composition $\alpha \beta$ where $\alpha$ is an involution which is not central and 
$\beta$ is a central element of order $3$. This implies that $G$ contains at most $71-18=53$ 
homologies. The right number of homologies cannot be obtained as the elements of order $4$ 
must be all homologies but the elements of order $3$ are divided into conjugacy classes of 
length $12$.

Assume that $G \cong SmallGroup(108,16)$. In this case $j=0$ and since $G$ contains just $6$ 
elements of order $3$ we get $k=0$ and $i=51$. Looking at the conjugacy classes of elements 
in $G$ we observe that $G$ contains $39$ involution, which are then homologies, and that 
just the elements of order $3$ or $9$ can be homologies as well, otherwise a direct 
computation shows that $i>51$. The unique possibility is that the normal subgroup $C_9$ of 
$G$ is generated by a homology which implies that $G$ fixes a point $P$ which is 
$\mathbb{F}_{71^2}$-rational but $P \not\in \cH_{72}$. By \cite[Page 6]{MZq} $C_9$ must be 
contained in $Z(G)=\{id\}$, contradiction.

Assume that $G \cong SmallGroup(108,17)$. Since $j=0$ we have that either $i=50$ or $i=51$. 
Denote by $T:=\{\alpha \in G \mid o(\alpha) \ne 2, \ and \ \alpha \ is \ a \ homology\}$. 
Then we require that either $|T|=23$ or $|T|=24$. Looking at the lengths of the conjugacy 
classes of elements in $G$ we observe that they are all even, and hence the case $|T|=23$ 
cannot occur. Since the elements of order $6$ are divided into three conjugacy classes each 
of length $18$, they are all of type (B1) since otherwise $|T|>20$ and $G$ contains no 
conjugacy classes of length $4$ or two conjugacy classes of length equal to $2$. Thus, the 
homologies of $T$ are all of order $3$. The unique possibility is that the remaining $24$ 
homologies of order $3$ are divided into $3$ distinct conjugacy classes, of length $12$, $6$ 
and $6$ respectively. This group cannot occur as a subgroup of $PGU(3,71)$ as every subgroup 
of type $C_3 \times C_3$ of $PGU(3,71)$ contains exactly $6$ homologies and $2$ elements of 
type (B1), a contradiction.

Assume that $G \cong SmallGroup(108,18)$. Denoting $T$ as before, we note that $|T|=48$ as 
the lengths of the conjugacy classes of elements in $G$ are all even. Since by a direct 
analysis of the conjugacy classes in $G$, $|T|\leq 31$, we have a contradiction.

Assume that $G \cong SmallGroup(108,19)$. Since $G$ fixes a points $P$ which is 
$\mathbb{F}_{71^2}$-rational but $P \not\in \cH_{72}$ and $Z(G)$ is cyclic and generated by 
a homology, we get that every $\beta \in G$ such that $\beta \not\in Z(G)$ and $|Z(G) \cap 
\langle \beta \rangle|>1$ is of type (B1). This implies that $48=|T| \leq 29$, a 
contradiction.

Assume that either $G \cong SmallGroup(108,20)$ or $G \cong SmallGroup(108,21)$. Here we 
need $|T|=48$. Since the length of the $3$ conjugacy classes of elements of order $9$ is 
$24$, we have that at most one of these conjugacy classes is given by homologies. Elements 
of order $3$ split into $4$ conjugacy classes each of length $2$. Since they give rise to a 
subgroup of $G$ which is isomorphic to $C_3 \times C_3$ we get that $6$ of them are 
homologies while $2$ of them are of type (B1) and that they fix the vertexes of a self-polar 
triangle. Since $G$ contains also homologies of order $2$ lying on the sides of the same 
triangles, $G$ must contain also $6$ homologies of order $6$, which are obtained as a 
composition $\alpha \beta$ of homologies having a common center and such that $o(\alpha)=2$, 
$o(\beta)=3$. The remaining $24$ elements of order $9$ cannot be homologies since otherwise 
$G$ must contain a subgroup isomorphic to $C_{18}$ generated by a homology of order $9$ and 
an involution. Since this implies that $|T|<51$ we have a contradiction.

Assume that $G \cong SmallGroup(108,22)$. Here $j=0$ and either $k=0$ with $|T|=48$, or 
$k=48$ with $i=46$ or $k=24$ with $i=47$.  The lengths of the conjugacy classes of elements 
of order which is not equal to $2$ are all even. This shows that the last case cannot occur. 
The elements of order $3$ generate $13$ subgroups of type $C_3 \times C_3$ sharing the same 
normal subgroup of $G$ of order $3$ which is generated by a homology, as different groups of 
type $C_3 \times C_3$ fix different $\mathbb{F}_{71^2}$-points. Since $G$ contains $26$ 
elements of type (B1) and the lengths of the conjugacy classes of elements of order $3$ are 
$(24,24,24,6,2)$, where the last class is given by two homologies, this case cannot occur.

Assume that $G \cong SmallGroup(108,23)$. In this case $|T|=32$ and $Z(G) \cong C_6$. Since 
$G$ contains two normal subgroup isomorphic to $C_6$, we note that just the central one is 
generated by a homology. Arguing as in the previous cases, observing that if $\alpha \not\in 
Z(G)$ and $\langle \alpha \rangle \cap Z(G)$ is not trivial then $\alpha$ is of type (B1), 
we get that $|T|<51$, a contradiction.

Assume that $G \cong SmallGroup(108,24)$. In this case $Z(G) \cong C_{18}$ is generated by a 
homology, $(i,j,k)=(51,0,0)$ and $|T|=44$. If $\alpha \not\in Z(G)$ and $\langle \alpha 
\rangle \cap Z(G)$ is not trivial then $\alpha$ is of type (B1). This yields $|T|<51$, a 
contradiction.

The cases $G \cong SmallGroup(108,25)$ and $G \cong SmallGroup(108,26)$ cannot occur as 
subgroups of $PGU(3,71)$. In fact, it is sufficient to observe that from the number and 
intersection structures of subgroups of type $C_3 \times C_3$ of $G$ at least a subgroup 
$C_3$ must be generated by a homology having the same center and axis of the generator of 
$Z(G)$. Since this implies that $C_3 \subseteq Z(G)$ but $|Z(G)|=2$ , we have a 
contradiction.

The case $G \cong SmallGroup(108,27)$ cannot occur as $G$ contains $55$ involutions implying 
that $|T|>55$.

Assume that $G \cong SmallGroup(108,28)$. This case can be discarded observing that $G$ 
contains $4$ subgroups isomorphic to $C_3 \times C_3$, and hence $8$ elements of type (B1). 
The lengths of the conjugacy classes of elements of order $3$ are $(6,6,6,6,2)$ where the 
last two elements are contained in $Z(G)$ and hence are homologies, a contradiction.

Assume that $G \cong SmallGroup(108,29)=C_9 \times C_3 \times C_2 \times C_2$. Here $G$ 
fixes the vertexes of a self-polar triangle and we need $|T|=52$. Since by direct checking 
$|T|=3+(2+2+2)+(2+2+2)+(6+6+6)=33$, this case cannot occur.

Assume that $G \cong SmallGroup(108,30)$ or $G \cong SmallGroup(108,31)$. Then $G$ fixes the 
vertexes of a self-polar triangle but $G$ is not abelian, a contradiction.

Assume that $G \cong SmallGroup(108,36)$. Then $(i,j,k)=(51,0,0)$ and $|T|=42$. Since having 
homologies of order $12$ would imply that $|T|>42$, we get that they are all of type (B1) 
and $Z(G)$ is generated by a homology of order $3$. Since there exist subgroups of type $C_3 
\times C_3$ not containing the central element of order $3$ this case cannot occur.

Assume that $G \cong SmallGroup(108,37)$. Since $(i,j,k)=(51,0,0)$ then the elements of 
order $6$ are homologies and there are just two elements of order $3$ are of type (B1). 
Since these two elements must be contained in a subgroup $C_6$ generated by a homology we 
have a contradiction.

Assume that $G \cong SmallGroup(108,38)$. This case cannot occur as $G$ contains too many 
elements of order $3$ fixing the vertexes of the same self-polar triangle.

Assume that either $G \cong SmallGroup(108,39)$ or $G \cong SmallGroup(108,40)$. In this 
case $k=0$ and $|T|=12$. Since $G$ contains $13$ subgroups of type $C_3 \times C_3$, and 
hence more than $12$ homologies, this case cannot occur.

Assume that $G \cong SmallGroup(108,41)$. In this case $Z(G) \cong C_3 \times C_3$ and the 
three involutions of $G$ are not central, a contradiction.

Assume that $G \cong SmallGroup(108,42)$. This group cannot occur as a subgroup of 
$PGU(3,71)$ as it fixes a point $P$ which is $\mathbb{F}_{71^2}$-rational but $P \not\in 
\cH_{72}$, its center is isomorphic to $C_3 \times C_3 \times C_2$ and $G$ contains too many 
elements of order $3$ fixing the same self-polar triangle.

Assume that $G \cong SmallGroup(108,43)$. This group cannot occur as a subgroup of 
$PGU(3,71)$ as $G$ cannot normalize $5$ groups of order $3$.

The case $G \cong SmallGroup(108,44)$ cannot occur as $G$ contains $55$ involution and hence 
$|T|>51$.

Assume that $G \cong SmallGroup(108,45)$. Since $G$ fixes the vertexes of a self-polar 
triangle but contains more than one subgroup isomorphic to $C_3 \times C_3$ this case cannot 
occur.

{\bf $|G|=112$.} From \cite[Thm. 2.7]{MZ} for $\sigma \in G$ we have $i(\sigma) \in 
\{0,2,72\}$. By direct checking with MAGMA the pairs $(i,j)$ with $\deg(\Delta)=3624=72 
\cdot i + 2 \cdot j$ and $0 \leq i+j \leq 111$, are $(i,j)=(49,48)$ and $(i,j)=(50,12)$. 
Since there are no groups of order $112$ having at least $49$ elements of order equal to 
$2$, $4$ and $8$ and either $48$ or $12$ elements of order in $\{7,14,16,28,56,112\}$, this 
case cannot occur.

{\bf $|G|=120$.} We write $\deg(\Delta)=3528=71 \cdot i + 2 \cdot j + 3 \cdot k$ where $4 
\leq i+j+k \leq 119$ as $G$ contains at least a subgroup of type $C_5$ and hence $j \geq 4$ 
from \cite[Thm. 2.7]{MZ}. By direct checking with MAGMA either $(i,j,k)=(47,3m,2(24-m))$ for 
some $m=2,\ldots,24$ or $(i,j,k)=(48,3m,2(12-m))$ for some $m=2,\ldots,12$. Denote as before 
$o_i$ the number of elements in $G$ or order $i$. Then from \cite[Thm. 2.7]{MZ}, 
$j=119-(o_2+o_3+o_4+o_6+o_8+o_{12}+o_{24})$. Forcing $3$ to divide $j$ we get that no groups 
of order $120$ satisfies the condition, a contradiction.

  {\bf $|G|=126$.} We write $\deg(\Delta)=3456=72 \cdot i + 2 \cdot j + 3 \cdot k$, for some 
$6 \leq i+j+k \leq 125$, as $G$ contains at least $6$ elements of order $7$ which implies $j 
\geq 6$ from \cite[Thm. 2.7]{MZ}. By direct checking with MAGMA, either 
$(i,j,k)=(45,3m,2(36-m))$ for some $m=2,\ldots,8$, or $(i,j,k)=(46,3m,2(24-m)$ for some 
$m=2,\ldots,24$, or $(i,j,k)=(47,3m,2(12-m))$ for some $m=2,\ldots,12$. There are $16$ 
groups of order $126$ up to isomorphism, we will denote them as $SmallGroup(126,\ell)$, with 
$\ell=1,\ldots,16$. Table \ref{tabella126} summarizes the possible values $m$, $i$, and $k$ 
corresponding to each $G \cong SmallGroup(126,\ell)$ in order to obtain a quotient curve 
$\cH_{72}/G$ of genus $7$, according to the number of elements of order in 
$\{7,14,21,42,63,126\}$, that is, the value of $j$.
   \begin{center}
\begin{table}[htbp]
\begin{small}
\caption{$G \cong SmallGroup(126,\ell)$, values for $i$, $k$ and $m$}\label{tabella126}
\begin{tabular}{|c|c|c|}
\hline $\ell$ & $m$ & $(i,k)$ \\
\hline\hline 1 & 6 & (45,69), (46,36), (47,12)\\
\hline  2 & 12 & (46,24), (47,0)\\
\hline 4 &  18 & (46,12) \\
\hline 5 & 18 & (46,12)\\
\hline 7 & 6 & (45,60), (46,36), (47,12)\\
\hline 8 & 12 & (46,24),(47,0)\\
\hline 9 & 6 & (45,60), (46,36),(47,12)\\
\hline 10 & 12 & (46,24),(47,0)\\
\hline 11 & 18 & (46,12)\\
\hline 13 & 18 & (46,12)\\
\hline 15 & 18 & (46,12)\\
\hline
\end{tabular}
\end{small}
\end{table}
    \end{center}
    If $G \cong SmallGroup(126,1)$ then $G$ contains just $2$ elements of order $2$, a 
contradiction.

If $G \cong SmallGroup(126,2)$ then $G$ contains $2$ elements of order $2$ and hence $k=0$ 
and $(i,j,k)=(47,36,0)$. The elements of order $18$ in $G$ are of type (B1) because 
otherwise $i>47$, and all the remaining elements are homologies. Every subgroup $C_9=\langle 
\alpha \rangle$ in $G$ is such that $\alpha^3$ generates the unique subgroup of order $3$ of 
$G$. Since there exists just one cyclic group of order $9$ generated by a homology having a 
fixed center, we have a contradiction.

Assume that $G \cong SmallGroup(126,4)$. This case can be excluded observing that $G$ 
contains just $2$ elements of order $3$, but we need $k=12$.

The case $G \cong SmallGroup(126,5)$ cannot occur as $G$ contains just $2$ elements of order 
$3$.

Assume that $G \cong SmallGroup(126,7)$. Since $G$ contains $44$ elements of order $3$, 
either $(i,k)=(46,36)$ or $(i,k)=(47,12)$. A contradiction is obtained observing that the 
lengths of the conjugacy classes of elements of order $3$ are $(14,14,14,2)$ and hence 
neither $k=36$ nor $k=12$ can be obtained.

Assume that $G \cong SmallGroup(126,8)$. Since the lengths of the conjugacy classes of 
elements of order $3$ are $(28,14,2)$ we get that $(i,k)=(47,0)$. The elements of order $6$ 
are all of type (B1) because otherwise $i>47$ and all the remaining elements are homologies. 
All these elements normalize a group isomorphic to $C_3$ which is generated by a homology, 
hence they share a fixed point and commute. This implies that every $\alpha \beta$ with 
$o(\alpha)=2$ and $o(\beta)=3$ generates an element of order $6$. Since there are just $42$ 
elements of order $6$ we have a contradiction.

Assume that $G \cong SmallGroup(126,9)$. A contradiction is obtained as before looking at 
the lengths of the conjugacy classes of elements of order $3$ with respect to the admissible 
values of $k$.

Assume that $G \cong SmallGroup(126,10)$. Since $G$ normalizes an involution, we get that 
$G$ fixes a point $P$ which is $\mathbb{F}_{71^2}$-rational but $P \not\in \cH_{72}$. Thus 
$G$ cannot contain Singer subgroups and $(i,j,k)=(47,36,0)$. Since $G$ contains a unique 
element of order $2$, which is central, arguing as before, we get that every element of 
order $6$ that does not generate $Z(G)$ is of type (B1) because otherwise it would be 
central itself. Thus $Z(G)$ is generated by a homology, the remaining elements of order $6$ 
are of type (B1) and the remaining elements of $G$ are all homologies. Also, $G$ contains 
$7$ elementary abelian group of type $C_3 \times C_3$ and hence at least $14$ elements of 
type (B1) of order $3$, a contradiction.

The cases $G \cong SmallGroup(126,13)$ and $G \cong SmallGroup(126,15)$ cannot occur as $G$ 
contains just $6$ and $8$ elements of order $3$ respectively.

  {\bf $|G|=128$.} We write $\deg(\Delta)=3432=72 \cdot i + 2 \cdot j$ with $0 \leq i+j \leq 
127$. By direct checking with MAGMA this implies that either $(i,j)=(46,60)$ or 
$(i,j)=(47,24)$. Denote by $o_i$ the number of elements in $G$ having order equal to $i$. 
Since a group of order $126$ with $o_{16}+o_{32}+o_{64}+o_{128}= 60$ and 
$127-(o_{16}+o_{32}+o_{64}+o_{128}) \geq 46$ or $o_{16}+o_{32}+o_{64}+o_{128}=24$ and 
$127-(o_{16}+o_{32}+o_{64}+o_{128}) \geq 47$ does not exist, we get a contradiction.

{\bf $|G|=135$.} Writing $\deg(\Delta)=3348=71 \cdot i + 2 \cdot j + 3 \cdot k$ for $4 \leq 
i+j+k \leq 134$, we get that either $(i,j,k)=(43,3m,2(42-m))$ for some $m=2,\ldots,7$, or 
$(i,j,k)=(44,3m,2(30-m))$ for some $m=2,\ldots,29$, or $(i,j,k)=(45,3m,2(17-m))$ for some 
$m=2,\ldots,17$, or $(i,j,k)=(46,3m,2(6-m))$ for some $m=2,\ldots,6$. Since by direct 
checking with MAGMA there are no groups of order $135$ with $o_3+o_9 \geq 43$ and $3$ 
dividing $134-(o_3+o_9)=j$, this case cannot occur.

{\bf $|G|=140$.} We write $\deg(\Delta)=3228=72 \cdot i + 2 \cdot j$ for some $4 \leq i+j 
\leq 139$. This yields either $(i,j)=(43,96)$ or $(i,j)=(44,60)$ or $(i,j)=(45,24)$. Since 
there are no groups of order $140$ with $o_2+o_4 \geq 43$ and $139-(o_2+o_4) \in 
\{96,60,24\}$, this case cannot occur.

{\bf $|G|=144$.} We write $\deg(\Delta)=3240=72 \cdot i + 2 \cdot j + 3 \cdot k$ for some $0 
\leq i+j+k \leq 143$. By direct checking with MAGMA we get that $(i,j,k)=(41,3m,2(48-m))$ 
for some $m=0,\ldots,6$, or $(i,j,k)=(42,3m,2(36-m))$ for some $m=0,\ldots,29$, or 
$(i,j,k)=(43,3m,2(24-m))$ for some $m=0,\ldots,24$, or $(i,j,k)=(44,3m,2(12-m))$ for some 
$m=0,\ldots,12$, or $(i,j,k)=(45,0,0)$.

Assume that $G \cong SmallGroup(144,1)$. In this case $G$ contains $72$ elements of order 
$16$ which are of type (B2) by \cite[Thm. 2.7]{MZ}. Thus $Z(G) \cong C_8$ is generated by a 
homology. Arguing as before, since $G$ has just one involution which is central, we get that 
$G$ contains at most $8+1$ homologies, a contradiction.

If $G \cong SmallGroup(144,2)$ then $G$ is cyclic and hence it contains $71$ homologies, a 
contradiction.

Assume that $G \cong SmallGroup(144,3)$. Since $j=0$ and the number of elements of order $3$ 
is $2$ we have $(i,j,k)=(45,0,0)$. This implies that the $96$ elements of order $9$, which 
are all conjugated, are of type (B1), and all the remaining elements but $2$ are homologies. 
The non-trivial elements of a group of order $3$ are of type (B1) and obtained as powers of 
a homology of order $6$, a contradiction.

Assume that $G \cong SmallGroup(144,4)$. In this case $Z(G) \cong C_2$, $(i,j,k)=(45,0,0)$ 
and, arguing as before, homologies can occur only of odd order. Then $i<45$, a 
contradiction.

Assume that $G \cong SmallGroup(144,\ell)$ for $\ell \in \{5,\ldots,27,31,\dots,50\}$. Every 
cyclic group that intersects non-trivially $Z(G) \cong C_8$ is generated by an element of 
type (B1), $G$ contains less than $41$ homologies, a contradiction.

Assume that $G \cong SmallGroup(144,\ell)$ with $\ell \in \{28,29,30,51\}$. Here $m=24$ and 
$G$ contains exactly one involution. Thus, $(i,j,k)=(47,72,0)$. A contradiction is obtained 
observing that $G$ has one involution and the remaining elements are divided into conjugacy 
classes of even length.

Assume that $G \cong SmallGroup(144,52)$. Since $G$ contains $8$ elements of order $3$ and 
$j=0$ we get $(i,j,k)=(45,0,0)$. A contradiction is obtained observing that $Z(G) \cong C_4$ 
is generated by a homology, $G$ fixes a point $P$ which is $\mathbb{F}_{71^2}$-rational but 
$P \not\in \cH_{72}$ and $i \leq 39$.

Arguing as in the previous case, all $G \cong SmallGroup(144,\ell)$ with $\ell \in 
\{53,54,\ldots,62\}$ can be excluded.

Assume that $G \cong SmallGroup(144,\ell)$ with $\ell \in 
\{63,\ldots,67,69,70,71,72,74,76,\ldots,79,84,86,90,92,\ldots,95,100,105,\ldots,108,134\}$. 
A contradiction is obtained observing that $G$ fixes the vertexes of a self-polar triangle 
and $G$ is not abelian, thus $G$ cannot occur as a subgroup of $PGU(3,71)$.

Assume that $G \cong SmallGroup(144,68)$. In this case $G$ contains $98$ elements of order 
$3$ divided into $4$ conjugacy classes of lengths $(28,28,28,2)$. The conjugacy classes of 
length $2$ cannot be given by elements of Singer subgroups, as $2$ divides the order of 
their normalizers. We note that $k \in \{96,72,48,24\}$ or $k=0$. Since $96,72,48,24$ are 
not divisible by $28$, we get $(i,j,k)=(45,0,0)$.  $G$ contains $16$ elementary abelian 
groups of order $9$: this gives rise to $16 \cdot 4$ homologies and so this case can be 
excluded.

If $G \cong SmallGroup(144,73)$ then $G$ has a normal involution and hence it fixes a point 
$P\not\in \cH_{72}$ which is $\mathbb{F}_{71^2}$-rational. As before we can obtain a 
contradiction noting that each cyclic group which intersects $Z(G)$ non-trivially is 
generated by an element of type (B1). This forces $i<45$, a contradiction.

If $G \cong SmallGroup(144,75)$ then either $G$ fixes the vertexes of a self-polar triangle 
without being abelian or $i<45$, a contradiction.

Assume that $G \cong SmallGroup(144,\ell)$ with $\ell \in 
\{80,81,82,83,85,89,97,99,115,121,122,124,126,128,\ldots,133,1135,\ldots,140\}$. Since $G$ 
normalizes an involution, it fixes a point $P \not\in \cH_{72}$ which is 
$\mathbb{F}_{71^2}$-rational, hence $(i,j,k)=(45,0,0)$. Every cyclic group intersecting 
non-trivially $Z(G)$ is generated by an element of type (B1). Thus $i<4$, a contradiction.

Assume that $G \cong SmallGroup(144,87)$. In this case $(i,j,k)=(45,0,0)$, $G$ contains $37$ 
involutions and $6$ homologies of order $3$ which are contained in a subgroup isomorphic to 
$C_3 \times C_3$. Then the number of the remaining homologies is $45-37-6=2$. Since there 
are no conjugacy classes for the remaining elements of length $2$ we get a contradiction.

If $G \cong SmallGroup(144,\ell)$ with $\ell \in \{88,96\}$ then $G$ contains more than $45$ 
involutions, a contradiction.

Assume that $G \cong SmallGroup(144,91)$. This group cannot occur as a subgroup of 
$PGU(3,71)$ as the number and the conjugacy class structures of elements of order $6$ are 
not compatible with the ones of elements of order $3$ and $2$.

Assume that $G \cong SmallGroup(144,98)$. The center $Z(G)$ is cyclic of order $2$, and 
hence as before the elements that generate a cyclic group containing $Z(G)$ are of type 
(B1). Since $G$ contains an elementary abelian group of order $9$ and $37$ involutions then 
$G$ contains at least $37+(2+2+2)$ homologies and we need to find two extra homologies. 
Since $C_4$ is not central it cannot be generated by a homology. We get a contradiction 
looking at the lengths of the conjugacy classes of the remaining elements.

If $G \cong SmallGroup(144,104)$ then $G$ contains at least $3$ homologies of order $2$, $6$ 
homologies of order $3$ and $6$ homologies which are contained in a subgroup of type $C_6 
\times C_6$. We need to find other extra $30$ homologies, but this number is not compatible 
with the lengths of the remaining conjugacy classes, as $4$ does not divide $30$.

Assume that $G \cong SmallGroup(144,123)$. To obtain the right homologies configuration we 
need $12$ homologies of order $4$ and $6$ homologies of order $6$. The center $Z(G)$ is 
cyclic of order $6$ and it is generated by a homology. Therefore the other elements $\alpha$ 
of order $6$ such that $\langle \alpha\rangle \cap Z(G)$ is an element of order $3$ in 
$Z(G)$ cannot be homologies because they are not central. A contradiction.

Assume that $G \cong SmallGroup(144,125)$. In this case $G$ contains $37$ homologies of 
order $2$. We need $45-37$ extra homologies, but this is not compatible with the lengths of 
the conjugacy classes of elements in $G$.

Assume that $G \cong SmallGroup(144,127)$. Other than the involution we need $G$ to contain 
$26$ homologies and hence every element of order $3$ must be a homology. Since they form an 
elementary abelian group $C_3 \times C_3$ which contains at least $2$ elements of type (B1) 
we get a contradiction.

{\bf $|G|=160$.} Writing $\deg(\Delta)=3048=72 \cdot i + 2 \cdot j$ for $4 \leq i+j \leq 
159$ we get that $(i,j) \in \{(39,120),(40,84),(41,48),42,12)\}$. Since by direct checking 
with MAGMA there are no groups of order $160$ with at least $39$ elements of order in 
$\{2,4,8\}$ and such that $3$ divides the number of elements of order in 
$\{5,10,16,20,32,40,80,160\}$, we get a contradiction.

{\bf $|G|=162$.} We first observe that $G$ must not contain elements of order in 
$\{27,54,81,162\}$, since such orders do not occur for elements in $PGU(3,71)$; see 
\cite[Lemma 2.2]{MZ}. Writing $\deg(\Delta)=3024=72 \cdot i + 3 \cdot j$ for $0 \leq i+j 
\leq 161$ we get that $(i,j) \in \{(37,120),(38,96),(39,72),(40,48),(41,24),(42,0)\}$. Since 
by direct checking with MAGMA there are no groups of order $162$ with at least $37$ elements 
of order in $\tau=\{2,6,9,18\}$ and such that either there no elements of order not in 
$\tau$, or their number is divisible by $24$, we get a contradiction.

{\bf $|G|=168$.} Writing $\deg(\Delta)=2952=72 \cdot i + 2 \cdot j+3 \cdot k$ for $6 \leq 
i+j+k \leq 159$ we get that
 $(i,j,k)=(36,3m,2(60-m))$ for some $m=2,\ldots,11$, or $(i,j,k)=(37,3m,2(48-m))$ for some $m=2,\ldots,34$, or $(i,j,k)=(38,3m,2(36-m))$ for some $m=2,\ldots,36$, or $(i,j,k)=(39,3m,2(24-m))$ for some $m=2,\ldots,24$, or $(i,j,k)=(40,3m,2(12-m))$ for some $m=2,\ldots,12$. By direct checking with MAGMA the condition that either $j=0$ or $3 \mid j$ and $i \geq 39$ yields $G \cong SmallGroup(168,\ell)$ with $\ell \in \{1,4,5,7,\ldots,18,23,\ldots,28,34,\ldots,38,42,43,46,\ldots,51,53,54,56\}$.
 
The cases $G \cong SmallGroup(144,\ell)$ with $\ell \in \{1,7,\ldots,11,23,42,43,47,49,53\}$ 
can be excluded because $2 \mid |Z(G)|$, which implies that $G$ fixes a point $P \not\in 
\cH_{72}$ which is $\mathbb{F}_{71^2}$-rational and hence that $k=0$, a contradiction to the 
required values $(i,j,k)$ for $\deg(\Delta)$.

The cases $G \cong SmallGroup(144,\ell)$ with $\ell \in 
\{4,5,12,\ldots,18,24,\ldots,28,34,37\}$ can be excluded using the fact that $2 \mid 
|Z(G)|$. Hence every element of even order generating a cyclic group containing an 
involution of the center cannot be an homology without being central. This yields $G$ to 
contain a small number of homologies.

The cases $G \cong SmallGroup(144,\ell)$ with $\ell \in \{35,36,38,46,50,56\}$ can be 
excluded because the number of involutions contained in $G$ is greater than the total number 
of homologies that $G$ can have according to the desired value of $\deg(\Delta)$.

The cases $G \cong SmallGroup(144,\ell)$ with $\ell \in \{48,54\}$ can be excluded because 
$G$ fixes the vertexes of a self-polar triangle but $G$ is not abelian.

{\bf $|G|=180$.} Writing $\deg(\Delta)=2808=72 \cdot i + 2 \cdot j+3 \cdot k$ for $4 \leq 
i+j+k \leq 179$ we get that $(33,6,140)$, or $(i,j,k)=(34,3m,2(60-m))$ for some 
$m=2,\ldots,25$, or $(i,j,k)=(35,3m,2(48-m))$ for some $m=2,\ldots,48$, or 
$(i,j,k)=(36,3m,2(36-m))$ for some $m=2,\ldots,36$, or $(i,j,k)=(37,3m,2(24-m))$ for some 
$m=2,\ldots,24$, or $(i,j,k)=(38,3m,2(12-m))$ for some $m=2,\ldots,12$.

The cases $G \cong SmallGroup(180,\ell)$ with $\ell \in \{1,10,14,23\}$ can be excluded as 
$5 \mid |Z(G)|$ yields $G$ to contain too many homologies.

The cases $G \cong SmallGroup(180,\ell)$ with $\ell \in \{2,3,5,9,15,16,17\}$ can be 
excluded since $2 \mid |Z(G)|$ and hence every element of even order generating a cyclic 
group containing an involution of the center cannot be an homology without being central. 
This forces $G$ to contain a small number of homologies.

The cases $G \cong SmallGroup(180,\ell)$ with $\ell \in \{7,11,25,27,29,30,34,36\}$ can be 
excluded since $G$ contains too many involutions and hence too many homologies.

The cases $G \cong SmallGroup(180,\ell)$ with $\ell \in \{6,8,13,18,20,28,31,32,33,35,37\}$ 
can be excluded as they cannot occur as subgroups of $PGU(3,71)$.

The cases $G \cong SmallGroup(180,\ell)$ with $\ell \in \{12,22\}$ can be excluded as $G$ 
fixes the vertexes of a self-polar triangle and it is not abelian.

The cases $G \cong SmallGroup(180,\ell)$ with $\ell \in \{19,21,26\}$ can be excluded since 
the value of $i$ is not compatible with the lengths of the conjugacy classes of elements in 
$G$.

The case $G \cong SmallGroup(180,23)$ can be excluded since $G$ acts on the vertexes of a 
self-polar triangle but its order does not divide the order of the maximal subgroup of 
$PGU(3,71)$ fixing globally a self-polar triangle; see \cite[Thm. A.10]{HKT}.

{\bf $|G|=189$.} Writing $\deg(\Delta)=2700=72 \cdot i + 2 \cdot j+3 \cdot k$ for $6 \leq 
i+j+k \leq 179$ we get that $(i,j,k)=(32,3m,2(66-m))$ for some $m=2,\ldots,24$, or 
$(i,j,k)=(33,3m,2(54-m))$ for some $m=2,\ldots,47$, or $(i,j,k)=(34,3m,2(42-m))$ for some 
$m=2,\ldots,42$, or $(i,j,k)=(35,3m,2(30-m))$ for some $m=2,\ldots,30$, or 
$(i,j,k)=(36,3m,2(18-m))$ for some $m=2,\ldots,18$, or $(i,j,k)=(37,3m,2(6-m))$ for some 
$m=2,\ldots,6$. Forcing $G$ to contain no elements of order in $\{27,189\}$, as they do not 
occur as orders of elements in $PGU(3,71)$ and the numerical conditions above, we get $G 
\cong SmallGroup(189,\ell)$ with $\ell \in \{3,\ldots,13\}$. We have that $k \geq 1$ and $G$ 
normalizes a cyclic group of order $7$, which is generated by an element of type (B2) by 
\cite[Thm. 2.7]{MZ}. Such a group fixes at least a point $P \not\in \cH_{72}$ which is 
$\mathbb{F}_{71^2}$-rational and this implies that $G$ cannot contain Singer subgroups. 
Therefore all these cases can be excluded.

{\bf $|G|=192$.} Writing $\deg(\Delta)=2664=72 \cdot i + 2 \cdot j+3 \cdot k$ for $0 \leq 
i+j+k \leq 179$ we get that $(i,j,k)=(31,3m,2(72-m))$ for some $m=0,\ldots,16$, or 
$(i,j,k)=(32,3m,2(60-m))$ for some $m=0,\ldots,39$, or $(i,j,k)=(33,3m,2(48-m))$ for some 
$m=0,\ldots,48$, or $(i,j,k)=(34,3m,2(36-m))$ for some $m=0,\ldots,36$, or 
$(i,j,k)=(35,3m,2(24-m))$ for some $m=0,\ldots,24$, or $(i,j,k)=(36,3m,2(12-m))$ for some 
$m=0,\ldots,12$, or $(i,j,k)=(37,0,0)$. Denote by $o_i$ the number of elements of order $i$ 
in $G$. Forcing, as for the previous case, $G$ to have no elements of order in 
$\{64,96,192\}$ and $o_3 \geq 24$ for $m=0$, $o_3 \geq 8$ for $m=32$, $o_3 \geq 16$ for 
$m=16$, we get that $G \cong SmallGroup(192,\ell)$ with $$\ell \in 
\{3,4,57,58,78,79,80,81,180,181,182,184,\ldots,199,201,\ldots,204,944,$$ 
$$955,956,992,1000,1001,1002,1008,1009,1020,\ldots,1025,1489,\ldots,1495,$$ 
$$1505,\ldots,1508,1509,1538,1540,1541\}.$$

The cases $G \cong SmallGroup(192,\ell)$ with $\ell \in 
\{57,58,78,79,80,81,186,187,203,204\}$ can be excluded as $G$ normalizes an involution and 
hence it cannot contain Singer subgroups, but $m \ne 0$.

in the remaining cases a contradiction is obtained observing that the elements of order $3$ 
form a unique conjugacy class in $G$ and that $o_3$ is not equal to any value which is 
admissible for $k$.

{\bf $|G|=210$.} Writing $\deg(\Delta)=2448=72 \cdot i + 2 \cdot j+3 \cdot k$ for $10 \leq 
i+j+k \leq 209$ we get that $(i,j,k)=(27,3m,2(84-m))$ for some $m=4,\ldots,14$, or 
$(i,j,k)=(28,3m,2(72-m))$ for some $m=4,\ldots,37$, or $(i,j,k)=(29,3m,2(68-m))$ for some 
$m=4,\ldots,60$, or $(i,j,k)=(30,3m,2(48-m))$ for some $m=4,\ldots,48$, or 
$(i,j,k)=(31,3m,2(36-m))$ for some $m=4,\ldots,36$, or $(i,j,k)=(32,3m,2(24-m))$ for some 
$m=4,\ldots,24$, or $(i,j,k)=(33,3m,2(12-m))$ for some $m=4,\ldots,12$. By direct checking 
with MAGMA $G \cong SmallGroup(210,\ell)$ with $\ell \in \{5,6,8,9,10,12\}$. All these cases 
can be excluded as $G$ cannot contain a sufficient number of homologies with respect to 
$\deg(\Delta)$.

{\bf $|G|=216$.}  Writing $\deg(\Delta)=2376=72 \cdot i +3 \cdot k$ for $0 \leq i+k \leq 
215$ we get that $(i,k) \in 
\{(26,168),(27,144),(28,120),(29,96),(30,72),(31,48),(32,24),(33,0)\}$. By direct checking 
with MAGMA, forcing $G$ to contain no elements of order in $\{27,54,108,216\}$, we get $G 
\cong SmallGroup(216,\ell)$ with $\ell \in \{12,\ldots,20,25,\ldots,177\}$.

Cases $G \cong SmallGroup(216,12)$ and $G \cong SmallGroup(216,13)$ cannot occur as $G$ 
fixes the vertexes of a self-polar triangle and $G$ is not abelian.

Assume that $G \cong SmallGroup(216,14)$. A contradiction can be obtained since $2 \mid 
|Z(G)|$ and hence every element of even order generating a cyclic group containing an 
involution of the center cannot be an homology without being central. This yields $G$ to 
contain a small number of homologies.

Cases $G \cong SmallGroup(216,15)$ and $G \cong SmallGroup(216,16)$ cannot occur as 
subgroups of $PGU(3,71)$, since $Z(G)$ must contain at least a subgroup of order $3$.

Case $G \cong SmallGroup(216,17)$ can be excluded as $G$ contains $16+11$ homologies, a 
contradiction.

Assume that $G \cong SmallGroup(216,18)$. In this case $G$ fixes the vertexes of a 
self-polar triangle $T$ and it is abelian. This implies that $G$ is contained is the 
stabilizer in $PGU(3,71)$ of $T$ which is abelian and isomorphic to $C_{72} \times C_{72}$. 
Suppose that $G$ is such that $g(\cH_{72} / G)=7$. Then the quotient curve $\cH_{72}/G$ 
would inherit the quotient group $C_{72} \times C_{72} / G$ which is abelian of order $24$. 
Since $PSL(2,8)$ does not contain abelian subgroup of order $216$, the claim follows.

Assume that $G \cong SmallGroup(216,19)$ or $G \cong SmallGroup(216,20)$. Since $G$ has a 
unique involution, as before, $G$ cannot contain Singer subgroups and hence $i=33$. Since 
$G$ contains $4$ distinct subgroups isomorphic to $C_3 \times C_3$, $G$ contains $16$ 
homologies of order $3$. The center $Z(G)$ is generated by a homology of order $24$. This 
proves that $G$ contains at least $39$ homologies, a contradiction.

Assume that $G \cong SmallGroup(216,25)$. As in the previous cases $i=33$, $G$ fixes a point 
$P \not\in \cH_{72}$ which is $\mathbb{F}_{71^2}$-rational, and $G$ contains no Singer 
subgroups. This group cannot occur as a subgroup of $PGU(3,71)$ as it contains $4$ groups 
isomorphic to $C_3 \times C_3$ sharing a fixed homology of order $3$ having center at $P$. 
Since they generate $8$ elements of order $3$ and of type (B1), we have a contradiction to 
the conjugacy class structure of elements of order $3$ in $G$.

The cases $G \cong SmallGroup(216,\ell)$ with $\ell \in 
\{26,27,30,\ldots,86,100,\ldots,152\}$ cannot occur as a subgroup of $PGU(3,71)$ as $G$ has 
to contain at least a central element of order $3$, while $3 \nmid |Z(G)|$.

The cases $G \cong SmallGroup(216,\ell)$ with $\ell \in \{28,29\}$ can be excluded since $G$ 
contains too many involutions with respect to the desired value of $i$.

Assume that $G \cong SmallGroup(216,87)$. In this case $G$ contains too many involution 
compared to the desired value for $i$.

The case $G \cong SmallGroup(216,88)$ cannot occur as a subgroup of $PGU(3,71)$ as $G$ 
contains $4$ elementary abelian groups of order $9$, which generate exactly $8$ elements of 
type (B1) and order $3$. This fact is in contradiction with the structure of the conjugacy 
classes of elements of order $3$ in $G$.

The cases $G \cong SmallGroup(216,\ell)$ with $\ell \in \{89,90,91,98\}$ cannot occur as $G$ 
must contain at least a central element of order $2$.

If $G \cong SmallGroup(216,92)$ then, looking at the structure of the elementary abelian 
subgroups of $G$, $G$ contains at least $12 \cdot 4 + 2+ 18 + 3$ homologies, a 
contradiction.

The cases $G \cong SmallGroup(216,\ell)$ with $\ell \in \{93,94,96,97,99,165\}$ can be 
excluded as $G$ contains too many involution with respect to the desired value for $i$.

If $G \cong SmallGroup(216,95)$ then $G$ contains at least $13 \cdot 4 + 2 + 18 + 3$ 
homologies, a contradiction.

If $G \cong SmallGroup(216,\ell)$ with $\ell \in 
\{153,154,156,157,158,159,160,161,162,163,164,166,167\}$, then $G$ contains too many 
homologies with respect to the desired value for $i$. This fact can be deduced looking at 
the number and the structure of subgroups isomorphic to $C_3 \times C_3$ and the number of 
involutions in $G$.

If $G \cong SmallGroup(216,\ell)$ with $\ell=168,\ldots,177$ then $G$ contains an abelian 
subgroup of order $54$, a contradiction to \cite[Thm. 11.79 ]{HKT} as $4g+4=32$.

If $G \cong SmallGroup(216,155)$ then either $(i,k)=(33,0)$ or $(i,k)=(32,24)$. In fact the 
lengths of the conjugacy classes $C_1$, $C_2$ and $C_3$ of elements of order $3$ are 
$|C_1|=8$ are $|C_2|=1$, $|C_3|=4$ and $\sigma \in C_1$ cannot be such that $i(\sigma)=3$ 
since it is a power of an element of order $6$. Assume that $(i,k)=(33,0)$. We note that the 
elements of order $3$ generate $13$ elementary abelian groups of order $9$. Since $k=0$ each 
of them must contain exactly two elements of type (B1) and $6$ elements of type (A). Since 
one elementary abelian group is given by $8$ elements of order $3$ which are all conjugated, 
this case cannot occur. Assume now that $(i,k)=(32,24)$. In particular the number of 
homologies which are not involutions must be equal to $32-9=23$. The elements of order $8$ 
or $12$, which are conjugated, cannot be homologies since their number is strictly greater 
than $23$. If elements of order $6$ or $4$ are generated by homologies then $G$ contains 
exactly either $18+2+9=29$ or $18+9=27$ homologies, a contradiction.

{\bf $|G|=224$.} We write $\deg(\Delta)=2280=72 \cdot i + 2 \cdot j$, for some $6 \leq i+j 
\leq 223$. By direct checking with MAGMA $(i,j) \in 
\{(27,168),(28,132),(29,96),(30,60),(31,24)\}$. Since for every $G \cong 
SmallGroup(224,\ell)$ with $\ell=1,\ldots,197$ we have that $j=96$, we need $(i,j)=(29,96)$. 
Forcing $G$ to contain at least $29$ elements of order a divisor of $72$, we have $G \cong 
SmallGroup(224,\ell)$ with
  $$
  \ell \in \{8,\ldots,11,13,14,16,19,\ldots,26,28,30,31,36,\ldots43,63,\ldots,67,71,
  \ldots 74,
  $$
  $$
  82,\ldots,87,91,\ldots,96,100,101,102,104,107,110,112,115,\ldots,121,128,
  $$
  $$
  129,130,131,137,\ldots,141,143,146,147,174,181,187\}\, .
  $$
  Assume that $G \cong SmallGroup(224,\ell)$ for some $\ell \in 
\{8,9,10,14,16,19,20,\ldots,24,26,36,37,38,41,63,\ldots,67,72,73,74,82,83\}$. Then $G$ fixes 
the vertexes of a self-polar triangle $T$ and $G$ is not abelian, a contradiction.

If $G \cong SmallGroup(224,\ell)$ for $\ell \in \{11,13,25,28,30,31\}$ then $G$ contains at 
least $33$, $33$, $17$, $17$, $15$, $29$ homologies, a contradiction.

If $G \cong SmallGroup(224,\ell)$ with $\ell \in \{39,40,42,43,143,146\}$ then $G$ 
normalizes an element $\gamma$ of type (B2). Denote by $P,Q,R$ the fixed points of $\gamma$; 
then they are $\mathbb{F}_{71^2}$-rational and $P \not\in \cH_{72}$, $Q,R \in \cH_{72}$. 
Thus, $G$ fixes $P$ and $G \leq \mathcal{M}_{71}=Stab_{PGU(3,71)}(P)$. If $\alpha \in G$ is 
a homology then either $\alpha(Q)=Q$ and $\alpha(R)=R$ which implies $\alpha \in 
Z(\mathcal{M}_{71})$ (see \cite[Page 6]{MZq}), or $o(\alpha)=2$ and $\alpha(R)=Q$ as 
homologies act with long orbits outside their center. Since the number of involutions in $G$ 
is not equal to $29$ and $|Z(G)|=2$, we have a contradiction.

If $G \cong SmallGroup(224,71)$ then $G$ contains too many involution, and hence too many 
homologies, with respect to $i=29$.

By direct checking with MAGMA, excluding the cases for which $G$ fixes the vertexes of a 
self-polar triangle but $G$ is not abelian, we just need to analyze $G \cong 
SmallGroup(224,\ell)$ with $\ell \in \{96,101,102,104,107,110,112,137\}$.

If $G \cong SmallGroup(224,\ell)$ with $\ell \in \{96,101,102,104,137\}$ then $G$ contains 
more than $29$ involution and hence more than $29$ homologies, a contradiction.

If $G \cong SmallGroup(224,107)$, $G \cong SmallGroup(224,110)$, $G \cong 
SmallGroup(224,112)$ then $G$ contains at least $23$, $19$, $15$ homologies. The extra 
homologies cannot be found because of the lengths of the conjugacy classes of elements in 
$G$.

{\bf $|G|=240$.} We write $\deg(\Delta)=2088=72 \cdot i + 2 \cdot j + 3 \cdot k$ where $4 
\leq i+j+k \leq 239$. By direct checking with MAGMA, Table \ref{tabella240} gives the 
complete list of possibilities for the triple $(i,j,k)$.
   \begin{center}
\begin{table}[htbp]
\begin{small}
\caption{Admissible values for $(i,j,k)$}\label{tabella240}
\begin{tabular}{|c|c|c|c|}
\hline $i$ & $j$ & $k$ & $m$ \\
\hline\hline 20 & 3m & 2(108-m) & m=2,3\\
\hline 21 & 3m & 2(96-m) & m=2,\ldots,26\\
\hline 22 & 3m & 2(84-m) & m=2,\ldots,49\\
\hline 23 & 3m & 2(72-m) & m=2,\ldots,72\\
\hline 24 & 3m & 2(60-m) & m=2,\ldots,60\\
\hline 25 & 3m & 2(48-m) & m=2,\ldots,48\\
\hline 26 & 3m & 2(36-m) & m=2,\ldots,36\\
\hline 27 & 3m & 2(24-m) & m=2,\ldots,24\\
\hline 28 & 3m & 2(12-m) & m=2,\ldots,12\\
\hline
\end{tabular}
\end{small}
\end{table}
   \end{center}
Also, Table \ref{casi240} gives the complete list of cases $G \cong SmallGroup(240,\ell)$ that have to be considered. 
For each case the required triple $(i,j,k)$ is described.
    \begin{center}
\begin{table}[htbp]
\begin{small}
\caption{Required values for $m$, $k$ and $i$ for $G \cong SmallGroup(240,\ell)$}\label{casi240}
\begin{tabular}{|c|c|c|c|}
\hline $\ell$ & $m$ & $k$ & $i$ \\
\hline\hline $1,\ldots,4$ & 72 & 0 & 23\\
\hline 32 & 64 & 16 & 23 \\
\hline 89,\ldots,91 & 16 & 16 & 27 \\
\hline 92, \ldots, 94 & 32 & 8 & 26\\
\hline 105, \ldots, 110 & 32 & 8$=o_3$ & 26\\
\hline 189 & 16 & 16 & 27\\
\hline 190 & 32 & 8 & 26\\
\hline 191 & 64 & 16 & 23\\
\hline 194 & 32 & 8$=o_3$ & 26\\
\hline 197,198 & 32 & 8$=o_3$ & 26 \\
\hline 204 & 64 & 16 & 27\\
\hline
\end{tabular}
\end{small}
\end{table}
    \end{center}
If $G \cong SmallGroup(224,1)$ then $G$ normalizes an element $\gamma$ of type (B2). Denote 
by $P,Q,R$ the fixed points of $\gamma$; then they are $\mathbb{F}_{71^2}$-rational and $P 
\not\in \cH_{72}$ but $Q,R \in \cH_{72}$. Thus, $G$ fixes $P$ and $G \leq 
\mathcal{M}_{71}=Stab_{PGU(3,71)}(P)$ and if $\alpha \in G$ is a homology then either 
$\alpha(Q)=Q$ and $\alpha(R)=R$ which implies $\alpha \in Z(\mathcal{M}_{71})$ (see 
\cite[Page 6]{MZq}), or $o(\alpha)=2$ and $\alpha(R)=Q$ as homologies act with long orbits 
outside their center. This implies that $G$ contains $39$ homologies, a contradiction.

If $G \cong SmallGroup(240,3)$ then $G$ contains just $7$ homologies, a contradiction.

The cases $G \cong SmallGroup(240,\ell)$ where $\ell \in \{32,89,\ldots,94,191,204\}$ can be 
excluded as $G$ must contain Singer subgroups but elements of order $3$ form a unique 
conjugacy class in $G$ with $o_3 \ne k$.

The cases $G \cong SmallGroup(240,\ell)$ where $\ell \in \{105,107,\ldots,110,189,190\}$ can 
be excluded as $G$ must contain Singer subgroups but all elements of order $3$ are 
normalized by at least one involution and hence they cannot be elements from Singer 
subgroups of $PGU(3,71)$.

The cases $G \cong SmallGroup(240,\ell)$ where $\ell \in \{106,194,197,198\}$ can be 
excluded as $G$ contains too many involution compared to the desired value for $i$.

 {\bf $|G|=243$.} We write $\deg(\Delta)=2052=72 \cdot i + 3 \cdot j$ for $0 \leq i+j \leq 
242$. Also, note that elements of order in $\{27,81,243\}$ cannot exist as they do not occur 
as element of $PGU(3,71)$, see \cite[Lemma 2.2]{MZ}. This yields $G \cong 
SmallGroup(243,\ell)$ with $\ell \in \{47,51,\ldots,67\}$. By direct checking with MAGMA 
$(i,j) \in \{(20,204), 
(21,180),(22,156),(23,132),(24,108),(25,84),(26,60),(27,36),(28,12)\}$ and hence $G$ must 
contain Singer subgroups of order $3$.

If $G \cong SmallGroup(243,47)$ then the elements of order $3$ of $G$ such that their 
contribution to $\deg(\Delta)$ is $3$ are either $0$ or $18$, since $243$ does not divide 
the order of the normalizer of a Singer subgroup of $PGU(3,71)$; see \cite[Thm. A.10]{HKT}. 
Since $18$ is not an admissible value for $k$ we have a contradiction.

If $G \cong SmallGroup(243,\ell)$ for $\ell \in \{51,52,53,54\}$ then, arguing as in the 
previous case, the normal subgroups of order $3$ cannot be Singer subgroups as $243$ does 
not divide the order of the normalizer of a Singer subgroup of $PGU(3,71)$. Since they 
generate an elementary abelian group of order $9$, $G$ contains $3$ normal homologies of 
order $3$, and hence $G$ fixes the vertexes of a self-polar triangle. This proves that $G$ 
cannot contain Singer subgroups, a contradiction.

Assume that $G \cong SmallGroup(243,\ell)$ where $\ell \in 
\{55,56,57,58,59,60,61,62,63,64,65,66\}$, then $G$ contains a central subgroup of order $3$ 
which is generated by a homology. In fact it is a power of more than $9$ different elements 
of order $9$ where at least two of them are of type (B1) and fix different self-polar 
triangles. Thus $G$ fixes a point $P$ off $\cH_{72}$ which is $\mathbb{F}_{71^2}$-rational 
and hence $j=0$, a contradiction.

If $G \cong SmallGroup(243,67)$ then $G$ is an elementary abelian $3$-group. As before $G$ 
contains no Singer subgroups as all the elements of order $3$ are normal in $G$ and $243$ 
does not divide the order of the normalizer of a Singer subgroup of $PGU(3,71)$, a 
contradiction.

{\bf $|G|=252$.} There exist $46$ different structures for groups of order $252$ up to 
isomorphism. We write $\deg(\Delta)=1944=72 \cdot i + 2 \cdot j + 3 \cdot k$ for some $6 
\leq i+j+k \leq 251$. By direct checking with MAGMA Table \ref{tabella252} summarizes the 
possibilities for $(i,j,k)$ with respect to $\deg(\Delta)$.
   \begin{center}
\begin{table}[htbp]
\begin{small}
\caption{Admissible values for $(i,j,k)$}\label{tabella252}
\begin{tabular}{|c|c|c|c|}
\hline $i$ & $j$ & $k$ & $m$ \\
\hline\hline 18 & 3m & 2(108-m) & m=2,\ldots,17\\
\hline 19 & 3m & 2(96-m) & m=2,\ldots,40\\
\hline 20 & 3m & 2(84-m) & m=2,\ldots,63\\
\hline 21 & 3m & 2(72-m) & m=2,\ldots,72\\
\hline 22 & 3m & 2(60-m) & m=2,\ldots,60\\
\hline 23 & 3m & 2(48-m) & m=2,\ldots,48\\
\hline 24 & 3m & 2(36-m) & m=2,\ldots,36\\
\hline 25 & 3m & 2(24-m) & m=2,\ldots,24\\
\hline 26 & 3m & 2(12-m) & m=2,\ldots,12\\
\hline
\end{tabular}
\end{small}
\end{table}
    \end{center}
All the cases $G \cong SmallGroup(252,\ell)$ with $\ell \in 
\{1,2,3,4,5,7,13,14,16,\ldots,19,22,24,43,45\}$ can be excluded as $G$ contains just one 
involution which is central. This implies that non-central elements of even order cannot be 
homologies. We get that the number of homologies in $G$ is strictly less that $18$, a 
contradiction.

If $G \cong SmallGroup(252,6)$ then, denoting by $\varphi$ the Euler totient function, $G$ 
contains exactly $1+2+2+2+\varphi(9) + \varphi(12)+ \varphi(18) + \varphi(36)=35$ 
homologies, a contradiction.

If $G \cong SmallGroup(252,8)$ then, $G$ contains too many involution with respect to the 
desired value of $i$, a contradiction.

If $G \cong SmallGroup(252,9)$ then $G$ fixes the vertexes of a self-polar triangle but $G$ 
is not abelian, a contradiction.

All the cases $G \cong SmallGroup(252,\ell)$ with $\ell \in 
\{10,11,21,23,27,29,31,32,35,38,39,40,42\}$ can be excluded as $G$ normalizes a cyclic 
subgroup of order $7$ which is generated by an element of type (B2), by \cite[Lemma 
2.2]{MZ}. Thus, if $\alpha \in G$ is a homology then either $\alpha \in Z(G)$ or 
$o(\alpha)=2$. This proves that $G$ cannot contain $i$ homologies.

All the cases $G \cong SmallGroup(252,\ell)$ with $\ell \in 
\{12,15,25,26,30,33,34,36,37,41,46\}$ can be excluded as $G$ contains too many homologies 
with respect to the desired value of $i$.

The case $G \cong SmallGroup(252,20)$ cannot occur as a subgroup of $PGU(3,71)$ as $Z(G)$ 
must be cyclic.

Assume that $G \cong SmallGroup(252,28)$. Then $G$ fixes a point $P\not\in \cH_{72}$ which 
is $\mathbb{F}_{71^2}$-rational as $G$ normalizes an involution. This implies that $G$ 
cannot contain Singer subgroups and hence all the elements of order $3$ are either of type 
(B1) or homologies. Since $G$ contains $7$ subgroups isomorphic to $C_3 \times C_3$, it 
contains at least $29$ homologies of order $3$. Also, $G$ contains $15$ involution, a 
contradiction.

The case $G \cong SmallGroup(252,44)$ cannot occur as a subgroup of $PGU(3,71)$ as at least 
an element of order $3$ must be central but $3 \nmid |Z(G)|$.

 {\bf $|G|=270$.} There exist $30$ structures for groups of order $252$ up to isomorphism. We 
write $\deg(\Delta)=1728=72 \cdot i + 2 \cdot j + 3 \cdot k$ for some $4 \leq i+j+k \leq 
269$. Forcing $G$ to contain no elements of order in $\{27,54,135,270\}$, we get $G \cong 
SmallGroup(270,\ell)$ with $\ell \geq 5$. By direct checking with MAGMA we have that Table 
\ref{tabella270} summarizes the possibilities for $(i,j,k)$ with respect to $\deg(\Delta)$.
  \begin{center}
\begin{table}[htbp]
\begin{small}
\caption{Admissible values for $(i,j,k)$}\label{tabella270}
\begin{tabular}{|c|c|c|c|}
\hline $i$ & $j$ & $k$ & $m$ \\
\hline\hline 14 & 3m & 2(120-m) & m=2,\ldots,15\\
\hline 15 & 3m & 2(108-m) & m=2,\ldots,38\\
\hline 16 & 3m & 2(96-m) & m=2,\ldots,61\\
\hline 17 & 3m & 2(84-m) & m=2,\ldots,84\\
\hline 18 & 3m & 2(72-m) & m=2,\ldots,72\\
\hline 19 & 3m & 2(60-m) & m=2,\ldots,60\\
\hline 20 & 3m & 2(48-m) & m=2,\ldots,48\\
\hline 21 & 3m & 2(36-m) & m=2,\ldots,36\\
\hline 22 & 3m & 2(24-m) & m=2,\ldots,24\\
\hline 23 & 3m & 2(12-m) & m=2,\ldots,12\\
\hline
\end{tabular}
\end{small}
\end{table}
   \end{center}
If $G \cong SmallGroup(270,\ell)$, $\ell=5,6,7$ then $G$ contains $31$, $7$, $7$ homologies, 
a contradiction.

The cases $G \cong SmallGroup(270,\ell)$ with $\ell \in \{8,9\}$ can be excluded as they do 
not occur as subgroups of $PGU(3,71)$.

The cases $G \cong SmallGroup(270,\ell)$ with $\ell \in \{10,11\}$ can be excluded as $G$ 
normalizes a cyclic subgroup of order $5$, which is generated by an element of type (B2), by 
\cite[Thm. 2.7]{MZ}, and hence the unique homologies of $G$ are either central or of order 
$2$. Since $i$ is not equal to the sum of the number of element of order $2$ and the number 
of the remaining non-trivial elements of $Z(G)$, this case cannot occur.

The cases $G \cong SmallGroup(270,\ell)$ with $\ell \in \{12,14,15,16,18,19,27,28,29\}$ can 
be excluded as $G$ contains too many involutions with respect to the desired number of 
homologies $i$. By direct checking all the remaining cases do not occur as subgroups of 
$PGU(3,71)$.

{\bf $|G|=280$.} We write $\deg(\Delta)=1608=72 \cdot i + 2 \cdot j$ for some $10 \leq i+j 
\leq 279$. By direct checking with MAGMA, $(i,j) \in 
T=\{(15,264),(16,228),(17,192),(18,156),(19,120),(20,84),(21,48),(22,12)\}$. Denote by $o_i$ 
the number of elements of order $i$ in $G$. Since there are no groups of order $280$ with 
$(o_5+o_7+o_{10}+o_{14}+o_{20}+o_{28}+o_{35}+o_{40}+o_{56}+o_{70}+o_{140}+o_{280}) \in 
\{12,48,84,120,156,192,228,264\}$ and $(o_2+o_4+o_8) \geq i_j$ where $i_j$ denotes the value 
of $i$ corresponding to $j= 
(o_5+o_7+o_{10}+o_{14}+o_{20}+o_{28}+o_{35}+o_{40}+o_{56}+o_{70}+o_{140}+o_{280})$ in $T$, 
this case cannot occur.

{\bf $|G|=288$.} As before we can exclude all $SmallGroup(288,\ell)$ with 
$\ell=1,\ldots,1045$ containing elements of order in $\{32,96,288\}$. Writing 
$\deg(\Delta)=1512=72 \cdot i + 2 \cdot j + 3 \cdot k$ with $0 \leq i+j+k \leq 287$ we get 
that $i=10,\ldots,21$. We observe that $G \cong SmallGroup(288,\ell)$ is such that $2 \mid 
|Z(G)|$ unless $\ell \in \{73,74,75,397,406,407,634,635,636\}$. If $2 \mid |Z(G)|$ then $G$ 
cannot contain Singer subgroups as they cannot be centralized by involutions; see \cite[Thm. 
A.10]{HKT}. This implies that $k=0$ and hence $(i,j) \in 
\{(14,258),(15,222),(17,150),(18,114),(19,78),(20,42),(21,6)\}$. Since there are no groups 
of order $288$ with $(o_{16}+o_{48}+o_{144}) \in \{258,222,186,150,114,78,42,6\}$ all these 
cases cannot occur.

Assume that $G \cong SmallGroup(288,\ell)$ with $\ell=73,74,75$. These cases can be excluded 
as $G$ does not contain sufficient elements of order $3$.

Assume that $G \cong SmallGroup(288,\ell)$ with $\ell \in \{397,406,407,634,635,636\}$. These cases can be excluded as the desired values for $k$ are not compatible with the lengths of the conjugacy classes of elements of order $3$ in $G$.

{\bf $|G|=315$.} We can write $\deg(\Delta)=1188=72 \cdot i + 2 \cdot j + 3 \cdot k$, where 
$10 \leq i+j+k \leq 314$. Denote by $o_i$ the number of elements of order $i$ in $G$. By 
\cite[Thm. 2.7]{MZ}, $j=o_5+p_{15}+o_{21}+o_{35}+o_{45}+o_{63}+o_{105}+o_{315}$, while $i 
\leq o_3+o_9$ and $k \leq o_3$.

If $G \cong SmallGroup(315,1)$ then $j=270$ and hence to obtain the right value for 
$\deg(\Delta)$ either $(i,k)=(8,24)$ or $(i,k)=(9,0)$. Since $G$ contains at least $6 \cdot 
7 + 2=44$ homologies, and elements of order $9$ must be homologies, we get a contradiction.

Assume that $G \cong SmallGroup(315,3)$. Then $j=270$ and hence either $(i,k)=(8,24)$ or 
$(i,k)=(9,0)$. Since $G$ contains at least $42$ homologies this case can be excluded.

The case $G \cong SmallGroup(315,4)$ can be excluded as it does not occur as a subgroup of $PGU(3,71)$. 

{\bf $|G|=320$.} We note that $G$ cannot contain elements of order in $\{32,64,160,320\}$ as 
they cannot be contained in $PGU(3,71)$. We write $\deg(\Delta)=1128=72 \cdot i + 2 \cdot j$ 
where $4 \leq i+j \leq 319$. Then by direct checking $(i,j) \in 
\{(7,312),(8,276),(9,240),(10,204),(11,168),(12,132),(13,96),(14,60),15,24)\}$. Since there 
are no groups of order $320$ with $o_2+o_4+o_8 \geq 7$ and $j=319-(o_2+o_4+o_8) \in 
\{312,276,240,204168,132,96,60,24\}$, these cases can be excluded.

{\bf $|G|=324$.} We write $\deg(\Delta)=1080=72 \cdot i + 3 \cdot k$ for $0 \leq i+k \leq 
324$. Then $$(i,k) \in \{(2,312),(3,288),(4,264),(5,240),(6216),(7,192),(8,168),(9,144),$$ 
$$(10,120),(11,96),(12,72),(13,48),(14,24),(15,0)\}.$$ Also, $G$ cannot contain elements of 
order in $\{27,54,81,108,162,324\}$ as they cannot be contained in $PGU(3,71)$.

Assume that $G \cong SmallGroup(324,\ell)$ with $\ell \in \{8,19,23,25\}$. Then $2$ divides 
$|Z(G)|$ and hence $G$ cannot contain Singer subgroups. This case can be excluded as at 
least an element of order $3$ must be contained in $Z(G)$ but $3 \nmid |Z(G)|$.

The cases $G \cong SmallGroup(324,\ell)$ with $\ell \in \{13,14,16,18,20,22,24\}$ can be 
excluded as they do not contain Singer subgroups, because $2 \mid |Z(G)|$ and, looking at 
the number and the intersections of subgroups isomorphic to $C_3 \times C_3$ we get that $G$ 
contains a number of homologies different from $i$.

The cases $G \cong SmallGroup(324,\ell)$ with $\ell \in \{46,47,48,83,87,89,101\}$ cannot 
occur as $G$ fixes the vertexes of a self-polar triangle but $G$ is not abelian.

All the remaining cases $G \cong SmallGroup(324,\ell)$ with $\ell \in 
\{49,50,51,\ldots,176\}$ can be excluded as they do not contain Singer subgroups and looking 
at the number and the intersections of subgroups isomorphic to $C_3 \times C_3$ we get that 
$G$ contains number of homologies different from $i$.

{\bf $|G|=336$.} We write $\deg(\Delta)=936=72 \cdot i + 2 \cdot j + 3 \cdot k$ with $6 \leq i+j+k \leq 335$. 
By direct checking, every $G \cong SmallGroup(336,\ell)$ with $\ell=1,\ldots,228$ and $\ell \ne 114$ has a unique subgroup of order $7$, which is hence characteristic. 

Case $G \cong SmallGroup(336,114)$ cannot occur as a subgroup of $PGU(3,71)$. Thus we can 
assume that $G$ has a unique Sylow $7$-subgroup, say $S$. Since a generator of $S$ is of 
type (B2) from \cite[Thm. 2.7]{MZ}, $G$ fixes a point $P\not\in \cH_{72}$ which is 
$\mathbb{F}_{71^2}$-rational and $G$ does not contain Singer subgroups. In all the remaining 
admissible cases we have that $(i,j) \in \{(9,144),(5,288)\}$. Assume that $j=288$. Then $G 
\cong SmallGroup(336,\ell)$ with $\ell \in 
\{56,74,75,78,\ldots,83,86,88,89,106,\ldots,113,115117,168,\ldots,170,190,192,204\}$.

Assume that $G \cong SmallGroup(336,56)$. Then $Z(G) \cong C_7$ and hence every element of 
order $3$ is a homology. Since this implies that $i>5$ we have a contradiction.

A contradiction to $i=5$ is obtained also for $G \cong SmallGroup(336,\ell)$ with $\ell \in 
\{74,75,78,80,86,88,109,110,111,112,113,115,117,168,170,204\}$.

If $G \cong SmallGroup(336,\ell)$ with $\ell \in \{79,81,82,83,89,106,107,108,169,192\}$ 
then $G$ fixes the vertexes of a self-polar triangle and $G$ is not abelian, a 
contradiction.

The case $G \cong SmallGroup(336,190)$ cannot occur as a subgroup of $PGU(3,71)$ as the involution obtained as a power of an element of order $14$ must be central.
Assume that $j=144$. Then $G \cong SmallGroup(336,\ell)$ with $\ell \in \{38,39,40,62,72,94,104\}$. Since in all these cases the number of homologies of $G$ is $3$ while $i=9$, we have a contradiction.

{\bf $|G|=360$.} We write $\deg(\Delta)=648=72 \cdot i + 2 \cdot j + 3 \cdot k$ where $4 
\leq i+j+k \leq 359$. Thus, $i \leq 8$. We note that for every $G \cong 
SmallGroup(360,\ell)$ with $\ell=1,\ldots,162$ and $\ell ne 51$, $G$ has a unique Sylow 
$5$-subgroup, which is generated by an element of type (B2) from \cite[Thm. 2.7]{MZ}. This 
implies, as before, that $G$ fixes a point $P$ off $\cH_{72}$ and every element of order $3$ 
or $9$ is a homology. Denote by $o_i$ the number of elements of order $i$ in $G$. Since for 
every $\ell \ne 51$ we have that $o_2+o_3+o_9 \geq 9$, while $i \leq 8$ we have a 
contradiction.

Thus, assume that $G \cong SmallGroup(336,51)$. Since $G$ contains $5$ homologies which 
belong to $Z(G)$ whereas all the elements of even order dividing $72$ are of type (B1), we 
have that $G$ contains exactly $5+20 \cdot i_1 + 40 \cdot i_2$ for some $i_1$ and $i_2$, a 
contradiction.

{\bf $|G|=378$.} We write $\deg(\Delta)=432=72 \cdot i + 2 \cdot j + 3 \cdot k$, where $6 
\leq i+j+k \leq 377$. Since $i \leq 5$, $G$ must contain at most $5$ involutions and it 
contains no elements of order in $\{27,54,189,378\}$. These conditions yield $G \cong 
SmallGroup(378,\ell)$ with $\ell \in 
\{2,6,16,,23,24,25,26,27,28,33,44,45,,46,48,52,54,60\}$. All these cases can be excluded 
observing that $G$ has a unique Sylow $7$-subgroup, which implies that every element of 
order $3$ or $9$ is a homology. This yields $i>5$, a contradiction.

{\bf $|G|=384$.} As before, $G$ contains no elements of order in $\{32,64,96,128,192,384\}$ 
and we write $\deg(\Delta)=360=72 \cdot i + 2 \cdot j + 3 \cdot k$ for $0 \leq i+j+k \leq 
383$. By direct checking with MAGMA $i \leq 4$ and $G \cong SmallGroup(384,\ell)$ with $\ell 
\leq 20169$ has $3$ involutions. Thus, $G$ contains at most one extra homology. If $G$ 
contains one extra homology then $G$ contains at least $2$ extra homology, since its order 
is at least equal to $3$. Thus $i=3$ and $G$ contains no homology of order different from 
$2$. When $o_2 eq 3$ a direct checking with MAGMA shows that $2 \mid |Z(G)|$, hence $G$ 
normalizes an involution and $G$ fixes a point $P\not\in \cH_{72}$ which is 
$\mathbb{F}_{71^2}$-rational. Also, $G$ normalizes a group of type $C_2 \times C_2$ and acts 
on the vertexes of a self-polar triangle. This implies that every element of order $3$ is an 
homology. Since $G$ contains at least $2$ elements of order $3$ we have a contradiction.

{\bf $|G|=405$.} Here $G$ contains no elements of order in $\{81,135,405\}$ and we write 
$\deg(\Delta)=108=72 \cdot i + 2 \cdot j + 3 \cdot k$, obtaining either 
$(i,j,k)=(0,3m,2(18-m))$ for some $m=2,\ldots,54$, or $(i,j,k)=(1,3m,2(6-m))$ for some 
$m=2,\ldots,6$. By direct checking with MAGMA if $G \cong SmallGroup(405,\ell)$ with $\ell 
\ne 15$ then $G$ has a unique Sylow $5$-subgroup. This implies that every element of order 
$3$ is a homology and then $i \geq 2$, a contradiction. Thus we assume that $G \cong 
SmallGroup(405,15)$. In this case $j=o_5+o_{15}+o_{27}+o_{45}=324$. Since $324$ is not an 
admissible value for $j$ we have a contradiction.
  \end{proof}
  \begin{question} From Theorem \ref{nonga} the Fricke-Macbeath curve $\mathcal{F}$ is an 
$\F_{71}^2$-maximal curve which is not a Galois subcover of the Hermitian curve $\cH_{72}$ 
over $\mathbb{F}_{71^2}$. It is still an open problem to determine whether $\mathcal{F}$ is 
covered by $\cH_{72}$ or not. A positive answer would provide with the first known example 
of a maximal curve which is covered but not Galois covered by the Hermitian curve over the 
finite field of maximality. Otherwise, $\mathcal{F}$ would be the first known example of an 
$\mathbb{F}_{p^2}$-maximal curve which is not covered by the Hermitian curve over a field of 
order $p^2h$ with $h\not\equiv 0\pmod{3}$.
   \end{question}

\section*{Acknowledgement} The authors would like to thank Massimo Giulietti for numerous 
discussions on the topic which led to significant improvements. The third author would like 
to thank Università degli Studi di Perugia, for the financial support received during his 
academic visit in January-February 2017; he also was partially suported by CNPq-Brazil 
(grant 308326/2014-8). This research was partially supported by Ministry for Education, University and
Research of Italy (MIUR) (Project PRIN 2012 \textit{Geometrie di Galois e strutture di incidenza}-Prot. N.
2012XZE22K$\textunderscore$005) and by the Italian National Group for Algebraic and Geometric Structures and their
Applications (GNSAGA-INdAM).

  \end{document}